\documentclass[11pt]{amsart}
\usepackage{amsfonts,latexsym,amsthm,amssymb,graphicx}
\usepackage[all]{xy}
\usepackage{hyperref}
\usepackage[usenames]{color} 
\usepackage{float}

\DeclareFontFamily{OT1}{rsfs}{}
\DeclareFontShape{OT1}{rsfs}{n}{it}{<-> rsfs10}{}
\DeclareMathAlphabet{\mathscr}{OT1}{rsfs}{n}{it}

\CompileMatrices

\newtheorem{theorem}{Theorem}[section]
\newtheorem{corollary}{Corollary}[theorem]
\newtheorem{lemma}[theorem]{Lemma}
\newtheorem{proposition}[theorem]{Proposition}
\newtheorem{conjecture}[theorem]{Conjecture}

\theoremstyle{definition}
\newtheorem{definition}[theorem]{Definition}
\theoremstyle{remark}
\newtheorem{remark}[theorem]{Remark}
\newtheorem{example}{Example}

\numberwithin{equation}{section}

\title[Log-Concavity for Banana and Necklace Graphs]{Log-Concavity of the Grothendieck Classes of Banana Graphs and Clasped Necklaces}
\author{Stephanie Chen}
\address{Department of Mathematics, Harvard University, Cambridge, Massachusetts, 02138,
U.S.A.}
\email{stephchen@math.harvard.edu}

\begin{document}

\begin{abstract}
The Grothendieck classes of melonic graphs satisfy a recursive relation and may be written as polynomials in the class of the moduli space $\mathcal{M}_{0,4}$ with nonnegative integer coefficients, conjectured to be log-concave. In this article, we investigate log-concavity and ultra-log-concavity for the Grothendieck class of banana graphs and the three families of polynomials involved in the recursive relation.  We prove that all four are log-concave, establishing the specific case of banana graphs for the log-concavity conjecture. We additionally introduce the infinite family of clasped necklaces, melonic graphs obtained by replacing an edge of a $2$-banana with a string of $m$-bananas. Using the recursive relation, we explicitly compute the classes of clasped necklaces and prove that they too are log-concave.   
\end{abstract}

\maketitle

{\let\thefootnote\relax\footnotetext{2025/3/20}}

\tableofcontents


\section{Introduction}\label{sec:intro}
In a series of papers beginning with \cite{Banana}, Aluffi and Marcolli study the algebro-geometric Feynman rule defined by associating to a Feynman graph $\Gamma$ the class of its graph hypersurface complement $\mathbb{U}(\Gamma)$ in the Grothendieck ring of varieties. Aluffi, Marcolli, and Qaisar apply this process to melonic graphs in \cite{MR4620363}, motivated by the appearance of melonic graphs as the dominating diagrams in the large $N$ limit for the $O(N)^3$ tensor model of Klebanov and Tarnopolsky \cite{Klebanov_2017}. After obtaining a recursive formula for computing the Grothendieck classes of melonic graphs, they proved that the classes may be written as polynomials in the class $\mathbb{S}:= [\mathcal{M}_{0,4}]$ with positive integer coefficients and further conjectured that these polynomials are log-concave, i.e. the coefficients satisfy the inequality $a_{i-1}a_{i+1}\leq a_i^2$. 

In this note, we prove the conjecture for two infinite families of melonic graphs and further investigate log-concavity and related properties of the polynomials which appear in the recursive formula. 

\subsection{Banana graphs}
In scalar field theory, one studies the Feynman graphs of the theory; these are finite graphs with both internal edges (edges which have both a starting and ending vertex) and external `half edges' (only one end attached to a vertex) that satisfy certain properties in relation to the Langrangian for the theory. To each Feynman graph of a given theory we associate a Feynman integral which captures the combinatorics of the graph. In sufficiently large spacetime dimension, the integral is defined on the complement of the vanishing locus of the Kirchhoff-Symanzik polynomial of the graph $\Psi_{\Gamma}(t)$, a homogeneous polynomial in the internal edges of the graph. Consequently, in \cite{Algebro_Geometric_Feynman_Rules}, Aluffi and Marcolli define the Grothendieck class of a graph with $n$ internal edges to be the class $[\mathbb{A}^n\setminus \hat{X}_{\Gamma}]$ in the Grothendieck ring of varieties $K_0(\text{Var}_{\mathbb{Q}})$, where $\hat{X}_{\Gamma}$ is the vanishing locus of $\Psi_{\Gamma}(t)$ in $\mathbb{A}^n$. This process is described briefly in section \ref{sec:Grothendieck_class_graph} below and in far more detail in \cite{marcolli2009feynmanintegralsmotives} and \cite{Feyman_Motives_Book}.  

Banana graphs are graphs consisting of two vertices and multiple edges between them. An $n$-banana, denoted $\Gamma_n$, is a banana graph with $n$ edges. The physical motivation behind studying banana graphs comes from scalar field theories with an interaction term of the form $\mathcal{L}_{\text{int}}(\phi) = \phi^n$, for which $\Gamma_n$ is a vacuum bubble Feynman graph, i.e. a Feynman graph without external edges. The associated Feynman integral takes the form (cf. \cite[Lem~1.8]{Banana})
\begin{equation}
\label{eq:integral_banana}
    U(\Gamma_n,p) = \frac{\Gamma((1-\frac{D}{2})(n-1)+1)C(p)}{(4\pi)^{\frac{(n-1)D}{2}}}\int_{\sigma_n}\frac{(t_1\ldots t_n)^{(\frac{D}{2}-1)(n-1)-1}\omega_n}{\Psi_{\Gamma}(t)^{(\frac{D}{2}-1)n}}
\end{equation}
where 
\begin{equation*}
    C(p) = (\sum_{v\in V(\Gamma)} \sum_{e\in E_{\text{ext}}(\Gamma_n), t(e)=v}p_e)^2
\end{equation*}
is a function of the external momentum, $D$ is the spacetime dimension for the theory, $\sigma_n$ is the simplex
\begin{equation*}
    \sigma_n = \{(t_1,\ldots, t_n) \in \mathbb{R}_+^n | \sum_i t_i = 1\},
\end{equation*}
$\omega_n$ is the volume form, and
\begin{equation*}
    \Psi_{\Gamma_n}(t) = \sum_{i=1}^n \prod_{j\neq i}t_j
\end{equation*}
is the Kirchhoff-Symanzik polynomial for $\Gamma_n$. In \cite{Banana}, the authors use the Cremona transformation to compute the Grothendieck class of banana graphs, obtaining the formula
\begin{equation}
\label{eq:banana_class_intro}
     \mathbb{B}_n := [\mathbb{A}^n\setminus \hat{X}_{\Gamma_n}] = \mathbb{T}\frac{\mathbb{T}^n - (-1)^n}{\mathbb{T}+1} + n\mathbb{T}^{n-1}
\end{equation}
where $\mathbb{T} := [\mathbb{G}_m]$ is the Grothendieck class of the multiplicative group. 

This paper focuses more generally on melonic graphs, discussed in detail in \cite{MR4620363}. Melonic graphs are constructed by iteratively replacing edges with a string of banana graphs. Using the deletion-contraction relation for the Grothendieck class of a graph derived in \cite{Deletion_Contraction}, the authors of \cite{MR4620363} obtained a recursive formula for the Grothendieck class of melonic graphs in terms of the classes $\mathbb{B}_n$ and the polynomials 
\begin{equation}
    f_n = \frac{\mathbb{T}^n - (-1)^n}{\mathbb{T}+1} \quad\quad g_n = n\mathbb{T}^{n-1} - \frac{\mathbb{T}^n - (-1)^n}{\mathbb{T}+1} \quad\quad h_n = \frac{\mathbb{T}^n + (-1)^n\mathbb{T}}{\mathbb{T}+1}. 
\end{equation}
Using this recursive formula, they proved that the Grothendieck class of a melonic graph may be written as a polynomial in $\mathbb{S} = \mathbb{T} - 1$ with nonnegative integer coefficients and conjectured that these polynomials are log-concave. 
 
\subsection{Log-concavity}
A sequence $a_0,\ldots, a_n$ of real numbers is \textit{log-concave} if it satisfies the property $a_k^2\geq a_{k-1}a_{k+1}$ for all $0<k<n$. 
One classical example of a log-concave sequence is given by the binomial sequence ${n \choose 0}, {n\choose 1}, \ldots, {n \choose n}$. A real polynomial $f(x) = a_nx^n+\ldots + a_0$ is said to be log-concave if its coefficients $a_0,\ldots, a_n$ form a log-concave sequence. By a result of Newton's, any real-rooted polynomial is log-concave \cite[Thm~2]{Stanley_LogConcave_Survey}, \cite[Lem~1.1]{branden2014unimodalitylogconcavityrealrootedness}. 

Log-concavity has been a central property of many conjectures, one of the most famous being the Heron-Rota-Welsh conjecture, which states the the characteristic polynomial of a matroid is log-concave. This conjecture was settled positively by Adiprasito, Huh, and Katz in \cite{Hodge_Combo_Adiprasito_Huh_Katz} using a combinatorial analogue of the Hodge-Riemann relations by studying the Chow ring of a matroid. A brief discussion of the proof is given in \cite{Hodge_Combo_Baker}. See \cite{Stanley_LogConcave_Survey}, \cite{Brenti_LG_survey}, \cite{branden2014unimodalitylogconcavityrealrootedness} for more general surveys on log-concavity, related properties, and methods to prove sequences have these properties. 

One of the main results we obtain in this paper is the following log-concavity result (cf. Theorem \ref{thm:f_g_h_b_LC}). 
\begin{theorem}
\label{thm:intro_recursive_LC}
For all $n\geq 0$, the Grothendieck classes
\begin{equation}
\label{eq:recursive_polys_intro}
    \begin{cases}
     \mathbb{B}_n = (\mathbb{S}+1)\frac{(\mathbb{S}+1)^n-(-1)^n}{\mathbb{S}+2} + n(\mathbb{S}+1)^{n-1} \\
     f_n(\mathbb{S}) = \frac{(\mathbb{S}+1)^n - (-1)^n}{\mathbb{S}+2}\\
     g_n(\mathbb{S}) = n(\mathbb{S}+1)^{n-1} - \frac{(\mathbb{S}+1)^n - (-1)^n}{\mathbb{S}+2}\\
     h_n(\mathbb{S}) = \frac{(\mathbb{S}+1)^n + (-1)^n(\mathbb{S}+1)}{\mathbb{S}+2}
    \end{cases}
\end{equation}
are all log-concave as polynomials in $\mathbb{S}$.  
\end{theorem}
In particular, this proves Conjecture 3.4 of \cite{MR4620363} for all banana graphs. As a follow-up to this result, we consider the property of being ultra-log-concave, i.e. that the sequence $\frac{(a_k)}{{n\choose k}}$ is a log-concave sequence. We fully establish for which $n\geq 1$ the four polynomials $\mathbb{B}_n, f_n(\mathbb{S}), g_n(\mathbb{S}), h_n(\mathbb{S})$ are respectively ultra-log-concave (cf. Proposition \ref{prop:f_g_h_b_ULC}): 
\begin{proposition}
 Let $n\geq 4$. If $n$ is odd, then $\mathbb{B}_n, f_n(\mathbb{S}), g_n(\mathbb{S}), h_n(\mathbb{S})$ are all not ultra-log-concave as polynomials in $\mathbb{S}$. If $n$ is even, only $\mathbb{B}_n$ is ultra-log-concave.     
\end{proposition}

We also establish log-concavity for another family of melonic graphs, denoted $G'_{m,n}$ for $m\geq 1, n\geq 2$, called \textit{clasped necklaces}. These are melonic graphs obtained from a $2$-banana by replacing a single edge with a string of $n-1$ $m$-bananas. The clasped necklace $G'_{3,7}$ is drawn below. Using the recursive formula, we obtain closed forms for the Grothendieck class of a clasped necklace (cf. Lemma \ref{lem:clasp_necklace_class_T}) and further prove that they are all log-concave as polynomials in $\mathbb{S}$ (cf. Proposition \ref{prop:clasp_necklace_LC}). 
\begin{figure}[H]
    \centering
    \includegraphics[width=0.4\linewidth]{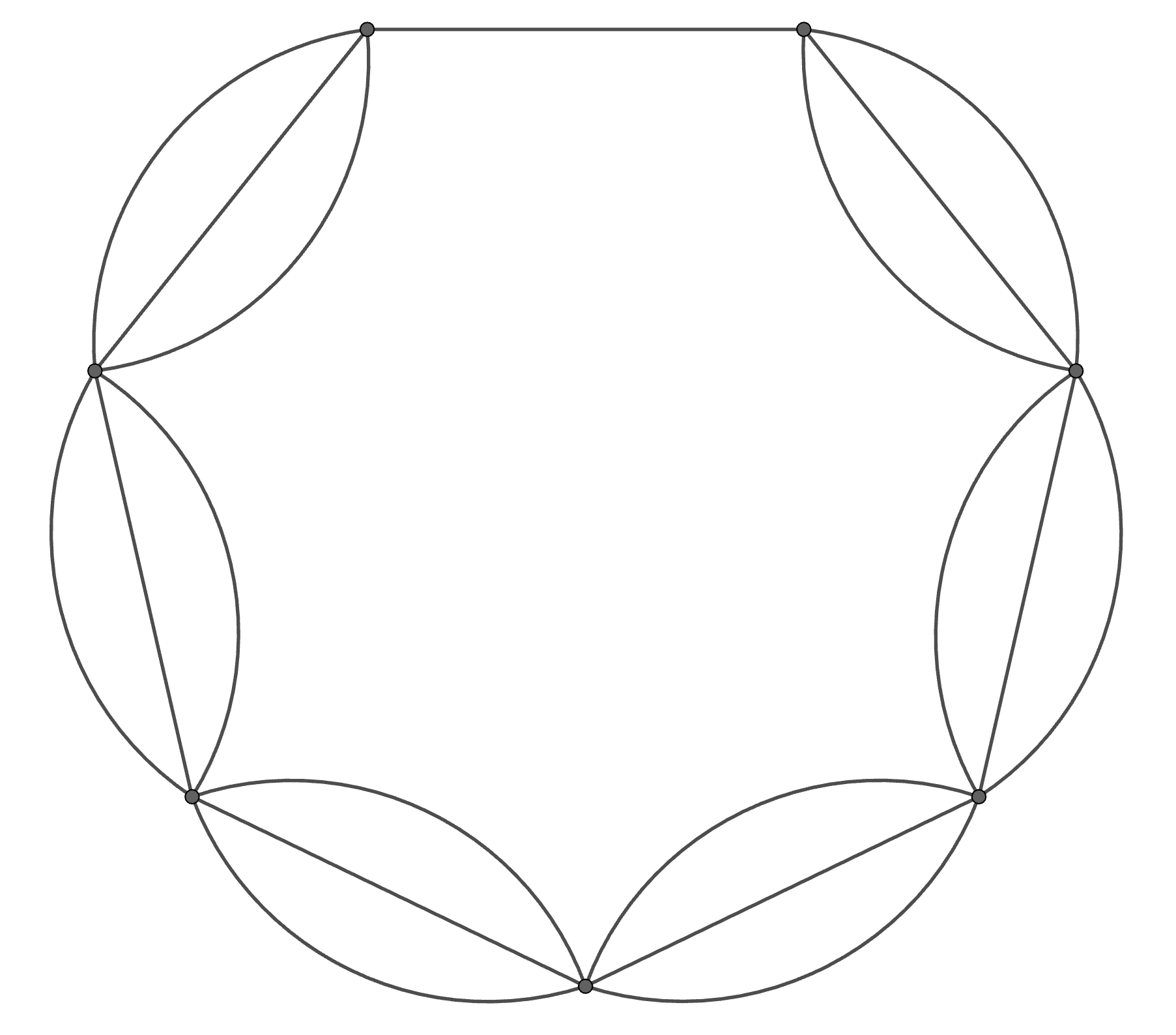}
    \caption{The clasped necklace $G'_{3,7}$, obtained by replacing an edge of a $2$-banana with a string of $6$ $3$-bananas. }
    \label{fig:intro_clasped_necklace_3_7}
\end{figure}

\subsection{Summary of paper}
The rest of this paper is organized as follows. In section 2, we recall the recursive formula for the Grothendieck class of a melonic graph and some key results on log-concave polynomials that we use in our proofs. We present results on the log-concavity and ultra log-concavity of the polynomials (\ref{eq:recursive_polys_intro}) in section 3. For clarity, most of the computations used in the proofs are deferred until section 5. Finally, in section 4, we derive the Grothendieck classes of clasped necklaces and establish their log-concavity as polynomials in $\mathbb{S}$. 

\textit{Acknowledgments.}
The results presented here come from an undergraduate research project conducted during the summer of 2023 as part of the California Institute of Technology's Summer Undergraduate Research Fellowship. The author is grateful to their mentors Paolo Aluffi and Matilde Marcolli for the initial project proposal, their endless support, and numerous helpful discussions throughout both the project and the preparation of this manuscript. 


\section{Grothendieck classes of melonic graphs}
\subsection{Grothendieck ring of varieties}
We first briefly recall the construction of the Grothendieck ring of varieties. Fix a field $k$ of characteristic 0. The Grothendieck group of $k$-varieties, denoted $K_0(\text{Var}_k)$, is the abelian group generated by isomorphism classes of $k$-varieties modulo the cut-and-paste relation
\begin{equation*}
    [X] = [Y] + [X\setminus Y]
\end{equation*}
for any variety $X$ and closed subvariety $Y$. 
The group $K_0(\text{Var}_k)$ admits a ring structure by defining multiplication of classes of varieties by $[X]\cdot [Y] = [X\times Y]$ and extending by linearity.  The resulting ring is called the Grothendieck ring of $k$-varieties. One may observe that any ring-valued invariant of algebraic $k$-varieties which is preserved by isomorphisms, satisfies the cut-and-paste relation, and is multiplicative over products must factor through $K_0(\text{Var}_k)$, hence the assignment of a variety $X$ to its Grothendieck class $[X]\in K_0(\text{Var}_k)$ may be viewed as a `universal Euler characteristic'. 

We fix notation for a couple key classes. First, we denote by $\mathbb{L}:= [\mathbb{A}^1]$ the class of the affine line.
We also define $\mathbb{T}:= [\mathbb{G}_m] = \mathbb{L}-1$ and $\mathbb{S}:= [\mathbb{P}^1 \setminus \{0,1,\infty\}] = \mathbb{T}-1$, the classes of the multiplicative group and the projective line with three points removed respectively. Note that $\mathbb{S}$ also arises as the class $[\mathcal{M}_{0,4}]$ of the moduli space of stable genus $0$ curves with $4$ marked points. 

\subsection{Grothendieck class of a graph}
\label{sec:Grothendieck_class_graph}
In \cite{Algebro_Geometric_Feynman_Rules}, Aluffi and Marcolli describe a procedure for constructing algebro-geometric Feynman rules, which assign to each Feynman graph an invariant in a commutative ring $\mathcal{R}$ satisfying certain conditions \cite[Def~2.2]{Algebro_Geometric_Feynman_Rules}. They defined a motivic Feynman rule to be one that factors through $K_0(\text{Var}_{\mathbb{Q}})$ via the Grothendieck class of a graph, to be defined below.

The main motivation came from Feynman rules in pertubative scalar field theory, where one considers one-particle-irreducible (1PI) Feynman graphs: finite edge-connected graphs where each vertex has valance equal to the degree of one of the monomials in the Langrangian for the theory, and each edge is either an internal edge connecting two vertices or an external edge with only one vertex. Fix $N:= \#E_{\text{ext}}(\Gamma), n:= \#E_{\text{int}}(\Gamma)$ as the number of external and internal edges respectively. Let $D\in \mathbb{N}$ be the spacetime dimension for the theory.
In the massless case, the Feynman rules dictate that we associate to $\Gamma$ a formal integral
\begin{equation}
\label{eq:integral_massless}
    U(\Gamma, p) = C\frac{\Gamma(n-\frac{D\ell}{2})}{(4\pi)^{\frac{D\ell}{2}}}\int_{\sigma_n}\frac{P_{\Gamma}(t,p)^{-n+\frac{D\ell}{2}}\omega_n}{\Psi_{\Gamma}(t)^{-n+(\ell+1)\frac{D}{2}}}
\end{equation}
where $C = \prod_{v\in V(\Gamma)} \lambda_v(2\pi)^D$ with $\lambda_v$ the coupling constant for the monomial of the Lagrangian with degree equal to the valence of $v$, $\omega_n$ is the volume form of the simplex $\sigma_n = \{t=(t_1,\ldots, t_n)\in \mathbb{R}^n_+ | \sum_i t_i = 1\}$, $\ell:=b_1(\Gamma)$ is the first Betti number of $\Gamma$, and $P_{\Gamma}(t,p), \Psi_{\Gamma}(t)$ are homogeneous polynomials in $t_1,\ldots, t_n$ defined as follows. 

Fix some ordering of the internal edges $E_{\text{int}}(\Gamma) = \{e_1,\ldots, e_n\}$. 
For each external edge $e\in E_{\text{ext}}(\Gamma)$, let $p_e$ be the incoming external momentum associated to $e$. For any cut-set $C$ dividing $\Gamma$ into two connected components $\Gamma = \Gamma_1\cup\Gamma_2$, we set 
\begin{equation*}
    s_C = \left( \sum_{v\in V(\Gamma_1)}P_v\right)^2 = \left(\sum_{v\in V(\Gamma_2)}P_v\right)^2 
\end{equation*}
where
\begin{equation*}
    P_v = \sum_{e\in E_{\text{ext}}(\Gamma), t(e) = v} p_e. 
\end{equation*}
The second Symanzik polynomial $P_{\Gamma}(t,p)$ is then given by
\begin{equation}
    P_{\Gamma}(t,p) = \sum_{C\subseteq \Gamma} s_C\prod_{e_i\in C}t_i
\end{equation}
where the sum is taken over all cut-sets of $\Gamma$. This is a homogeneous polynomial of degree $\ell + 1$. 

The Kirchhoff-Symanzik polynomial $\Psi_{\Gamma}(t)$ is defined by 
\begin{equation}
\label{eq:KirchhoffSymanzik_poly}
    \Psi_{\Gamma}(t) = \sum_{T\subseteq \Gamma}\prod_{e_i \notin E(T)}t_i
\end{equation}
where the sum is taken over all spanning trees of $\Gamma$. Note that $\Psi_{\Gamma}(t)$, which is homogeneous of degree $\ell$, may be defined for any finite graph $G$. The Kirchhoff-Symanzik polynomial may also be defined as the determinant of a $\ell\times \ell$ matrix $M_{\Gamma}(t)$ associated to the graph; see section 1.2 of \cite{Banana} for more details. 

In the massive case, the parametric Feynman integral takes the form
\begin{equation}
\label{eq:integral_massive}
    U(\Gamma, p) = C\frac{\Gamma(n-\frac{D\ell}{2})}{(4\pi)^{\frac{D\ell}{2}}}\int_{\sigma_n}\frac{V_{\Gamma}(t,p)^{-n+\frac{D\ell}{2}}\omega_n}{\Psi_{\Gamma}(t)^{\frac{D}{2}}}
\end{equation}
with the full formula for $V_{\Gamma}(t,p)$ given by (3.15) in the proof of \cite[Prop~3.1.7]{Feyman_Motives_Book}. 
When $D\geq \frac{2n}{\ell}$, the integrals (\ref{eq:integral_massless}) and (\ref{eq:integral_massive}) are defined on the complement of the vanishing locus of $\Psi_{\Gamma}(t)$. In the massless case, as the polynomials are homogeneous, we are interested in the complement in $\mathbb{P}^{n-1}$ of the graph hypersurface 
\begin{equation}
\label{eq:proj_graph_hyper}
    X_{\Gamma} = \{t\in \mathbb{P}^{n-1}|  \Psi_{\Gamma}(t)=0\}.
\end{equation}
In the massive case, we consider instead the complement in $\mathbb{A}^n$ of the affine hypersurface
\begin{equation}
\label{eq:affine_graph_hyper}
    \hat{X}_{\Gamma} = \{t\in \mathbb{A}^n | \Psi_{\Gamma}(t)=0\}. 
\end{equation}
In both settings, the assignment of $U(\Gamma, p)$ to $\Gamma$ satisfies a multiplicative property: if $\Gamma$ is a disjoint union $\Gamma = \Gamma_1\cup \Gamma_2$, then 
\begin{equation*}
    U(\Gamma, p) = U(\Gamma_1, p_1)U(\Gamma_2, p_2). 
\end{equation*}
To account for both the massless and massive cases, as well as preserve the multiplicative property (see \cite[Lem~~2.3]{Algebro_Geometric_Feynman_Rules}), Aluffi and Marcolli defined a motivic Feynman rule to be an assignment of a Feynman graph $\Gamma$ to an invariant $U(\Gamma)\in \mathcal{R}$ lying in a fixed commutative ring $\mathcal{R}$ such that $\Gamma \mapsto U(\Gamma)$ factors through a ring homomorphism $K_0(\text{Var}_{\mathbb{Q}})\to \mathcal{R}$ via $\Gamma \mapsto [\mathbb{A}^n\setminus \hat{X}_{\Gamma}]$. This leads us to the definition of the Grothendieck class of a finite graph in greater generality.

\begin{definition}
Let $G$ be a finite graph.
The Grothendieck class of $G$ is the class 
\begin{equation}
    \mathbb{U}(G) = [\mathbb{A}^n\setminus \hat{X}_G]\in K_0(\text{Var}_{\mathbb{Q}})
\end{equation}
with $n$ the number of edges of $G$, and $\Psi_G(t), \hat{X}_G$ as defined in (\ref{eq:KirchhoffSymanzik_poly}), (\ref{eq:affine_graph_hyper}). 
\end{definition}

\begin{example}[Banana graphs]
A banana graph is a graph consisting of two vertices and multiple edges. 
We denote by $\mathbb{B}_n$ the Grothendieck class of the banana graph with $n>0$ edges. Then,
(cf. \cite[Thm~3.10]{Banana}, \cite[Lem~2.6]{Algebro_Geometric_Feynman_Rules}, \cite[Cor~5.6]{Deletion_Contraction})
 \begin{equation}
 \label{eq:banana_grothendieck_class}
     \mathbb{B}_n = \mathbb{T}\frac{\mathbb{T}^n - (-1)^n}{\mathbb{T}+1} + n\mathbb{T}^{n-1} = (\mathbb{S}+1)\frac{(\mathbb{S}+1)^n - (-1)^n}{\mathbb{S}+2} + n(\mathbb{S}+1)^{n-1} 
 \end{equation}
In particular, letting $L$ be the $1$-banana, we see that $\mathbb{U}(L) = \mathbb{L}$. Banana graphs and their Grothendieck classes are the focus of \cite{Banana}, in which lays much of the foundation for \cite{MR4620363} due to the role of banana graphs as building blocks for melonic graphs.
\end{example}

The Grothendieck class of a graph satisfies the following properties with respect to elementary graph operations (cf. \cite[\S 2]{Algebro_Geometric_Feynman_Rules}): 
 \begin{itemize}
 \item For two disjoint graphs $\Gamma_1,\Gamma_2$ or two graphs that share only a single vertex, 
 \begin{equation*}
     \mathbb{U}(\Gamma_1 \cup \Gamma_2) = \mathbb{U}(\Gamma_1)\mathbb{U}(\Gamma_2)
 \end{equation*}
 \item If $\Gamma$ is obtained by connecting two disjoint graphs $\Gamma_1,\Gamma_2$ with a single edge, then 
 \begin{equation*}
     \mathbb{U}(\Gamma) =  [\mathbb{A}^1]\mathbb{U}(\Gamma_1)\mathbb{U}(\Gamma_2)
 \end{equation*}
     \item If $\Gamma'$ is obtained from $\Gamma$ by attaching an edge to a vertex, then
     \begin{equation*}
         \mathbb{U}(\Gamma') = [\mathbb{A}^1]\mathbb{U}(\Gamma)
     \end{equation*}
     \item If $\Gamma'$ is obtained from $\Gamma$ by splitting an edge, then 
     \begin{equation*}
         \mathbb{U}(\Gamma') = [\mathbb{A}^1]\mathbb{U}(\Gamma)
     \end{equation*}
 \end{itemize}
 Additionally in \cite{Deletion_Contraction}, Aluffi and Marcolli derived a deletion-contraction relation 
for the Grothendieck classes of graphs. Let $\Gamma$ be a finite graph. Fix an edge $e$ in $\Gamma$. We denote by $\Gamma\setminus e$ the graph obtained from $\Gamma$ by deleting the edge $e$ and $\Gamma / e$ the graph obtained by contracting (i.e. identifying the two vertices of) the edge $e$. Let $\Gamma_{ne}$ be the graph obtained by replacing $e$ by $n$-multiple edges, i.e. an $n$-banana for $n>0$.  In particular, $\Gamma_e = \Gamma$ and $\Gamma_{0e}=\Gamma \setminus e$. We say that $e$ is a bridge if the removal of $e$ disconnects the graph and a looping edge if $e$ is attached to only one vertex. In the case that $e$ is neither a bridge nor a looping edge, we have the following deletion-contraction relation (cf. \cite[Thm~5.3]{Deletion_Contraction})
\begin{equation}
\label{eq:deletion_contraction}
    \mathbb{U}(G'_{ne}) = f_{n}\mathbb{U}(G') + g_{n}\mathbb{U}(G'/e) + h_{n}\mathbb{U}(G' \setminus e)
\end{equation}
where (cf. \cite[Cor~5.7]{Deletion_Contraction})
\begin{equation}
\label{eq:recursive_poly}
    f_n = \frac{\mathbb{T}^n - (-1)^n}{\mathbb{T}+1} \quad\quad g_n = n\mathbb{T}^{n-1} - \frac{\mathbb{T}^n - (-1)^n}{\mathbb{T}+1} \quad\quad h_n = \frac{\mathbb{T}^n + (-1)^n\mathbb{T}}{\mathbb{T}+1}
\end{equation}
We note that notationally we have swapped $g_n,h_n$ in comparison to \cite{Deletion_Contraction} in order to match the notation used in \cite{MR4620363}. 
These properties are useful for computing the Grothendieck class of a graph, as we demonstrate below. 


\subsection{Melonic graphs}
\label{sec:melonic_graphs}
The authors of \cite{MR4620363} were interested in melonic graphs due to their role as the dominating diagram in the CTKT models of \cite{Bonzom_2011} in the large $N$ limit. Melonic graphs are graphs that may be constructed by starting with a single edge (and two vertices) and repeatedly replacing edges with strings of banana graphs.  More precisely, a melonic construction is defined as follows.
\begin{definition}
    A \textit{$n$-stage melonic construction} is a list $T = (t_1,\ldots, t_n)$ of tuples $t_s = (b_s, p_s, k_s)$ such that
\begin{enumerate}
    \item $b_s=(a_1,\ldots, a_{r_s})$ is a non-empty tuple of positive integers with length $|b_s|:= r_s\geq 1$, 
    \item $p_s$ is an integer with $p_1=0$ and $0<p_s<s$ for $s>1$,
    \item $k_s$ is an integer satisfying $1 \leq k_s \leq |b_{p_s}|$,
    \item and for all $t_i = ((a_1,\ldots,a_{r_i}), p_i, k_i)$, $i=1,\ldots, n$, $j = 1, \ldots, r_i$, at most $a_j$ tuples $t_s=(b_s, p_s, k_s)$ are such that $p_s=i, k_s=j$. 
\end{enumerate}
An $n$-stage melonic construction is also said to have \textit{depth} $n$.
\end{definition}

Each tuple in the melonic construction $T$ can be interpreted as a step in the corresponding melonic graph construction. One starts with a single edge connecting two vertices in stage $0$. The first tuple $t_1 = ((a_1,\ldots, a_{r_1}), 0, 1)$ represents replacing the single edge of the first (and only) banana in stage $0$ with a string of $r_1$ bananas, where the $i$th banana has $a_i$ edges ($1 \leq i \leq r_1$). This is stage $1$. Continuing, the tuple $t_s = ((a_1,\ldots, a_{r_s}), p_s, k_s)$ corresponds to replacing a single edge of the $k_s$-th banana constructed in stage $p_s$ with a string of $r_s$ bananas, again with the $i$-th banana being an $a_i$-banana ($1 \leq i \leq r_s$). Stage $p$ is the resulting graph after applying $p$ tuples. 

A \textit{melonic graph} is a graph resulting from a melonic construction in this manner. As noted in \cite{MR4620363}, two different melonic constructions may determine the same melonic graph (up to graph isomorphism), in which case the constructions are said to be equivalent. 
If a melonic construction $T$ further satisfies the condition that 
for all $t_i = ((a_1,\ldots, a_{r_i}), p_i, k_i)$, $i=1,\ldots, n$, if $a_j = 1$ then $k_s\neq j$ for all $s$ such that $p_s=i$, it is said to be \textit{reduced}. In terms of the stage by stage construction, this is equivalent to no $1$-banana constructed after stage 0 being replaced in a later stage. All melonic graphs admit a reduced construction (\cite[Lem~2.1]{MR4620363}). 

 The authors of \cite{MR4620363} developed a recursive formula for calculating the Grothendieck class of a melonic graph, which we review here. 
Given a reduced melonic construction $T$ with a total of $n$ stages, the following cases are considered. First, if $n=1$, i.e. $T$ corresponds to a graph that is a string of bananas $t_1 = ((a_1,\ldots, a_r), 0, 1)$, the Grothendieck class is simply a product of banana classes, $\mathbb{U}(\Gamma) = \prod_{i=1}^r \mathbb{B}_{a_i}$. 
The second case is where $n>1$ and $t_n = ((a), p, k)$ consists of adding a single $a$-banana. An equivalent melonic construction of depth $n-1$ may be obtained by omitting the last stage $t_n$ and altering $t_p$ such that the $k$th banana added in stage $p$ has $a-1$ extra edges than originally prescribed. In the third case, we consider a melonic construction of depth $n>1$ such that $t_n = ((1,\ldots, 1), p, k)$ corresponds to adding a string of $r$ $1$-bananas. In terms of graph operations, this is equivalent to splitting an edge of the $k$-th banana constructed in stage $p$ a total of $r-1$ times, hence letting $T'$ denote the construction given by omitting the last stage, we obtain $\mathbb{U}(T) = \mathbb{L}^{r-1}\mathbb{U}(T')$. 

The final case involves the deletion-contraction relation derived in \cite[Thm~5.3]{Deletion_Contraction}. We consider a melonic construction $T$ of depth $n>1$ such that $t_n = ((a_1,\ldots, a_r), p, k)$ corresponds to adding more than one banana, at least one of which has more than one edge. We denote the resulting graph by $G$. Letting $1 \leq m\leq r$ be such that $a_m$ is the maximum number of edges of the bananas added in stage $n$, we may consider the graph $G'$ giving by replacing the $a_m$-banana with a $1$-banana. Letting $e$ denote the single edge of the $1-$banana, $G = G'_{me}$ is the graph obtained from $G'$ by replacing $e$ with $m$-parallel edges. Note that as $T$ is reduced and $n>1$, $e$ is neither a bridge nor a looping edge of $G'$. Applying the deletion-contraction relation (\ref{eq:deletion_contraction}) then gives 
\begin{equation}
    \mathbb{U}(G'_{a_me}) = f_{a_m}\mathbb{U}(G') + g_{a_m}\mathbb{U}(G'/e) + h_{a_m}\mathbb{U}(G' \setminus e).
\end{equation}
Rewriting in terms of Grothendieck classes of melonic graphs gives
\begin{equation}
\label{eq:recursive_melonic}
    \mathbb{U}(T) = f_{a_m}\mathbb{U}(T') + g_{a_m}\mathbb{U}(T'') + \left(\prod_{i=1}^{m-1}\mathbb{B}_{a_i} \right)\left(\prod_{i=m+1}^r \mathbb{B}_{a_i} \right)h_{a_m}\mathbb{U}(T''')
\end{equation}
where $T'$ is a melonic construction for $G'$, $T''$ is a melonic construction for $G'/e$, and $T'''$ is a melonic construction obtained from $T$ be omitting $t_n$ and decreasing the order of the $k$th banana in $t_p$ by $1$. Note that the graph $G'\setminus e$ corresponds to the melonic graph obtained by taking the melonic graph constructed from $T'''$ and attaching  two strings of bananas, each at a single vertex. 

From this recursive formula for the Grothendieck class, the authors of \cite{MR4620363} proved that for any melonic construction $T$, $\mathbb{U}(T)$ could be written as a polynomial in $\mathbb{S}$ with nonnegative integer coefficients \cite[Cor~3.4]{MR4620363}. They also made the following conjecture.
\begin{conjecture}
\label{conj:log_concave_melonic}
Let $T$ be a melonic construction and $\mathbb{U}(T) = a_0 + a_1\mathbb{S}^1 + \ldots + a_n\mathbb{S}^n$ its Grothendieck class. Then, the sequence $a_0,a_1,\ldots, a_n$ is log-concave.   
\end{conjecture}

\subsection{Log-concavity and related properties}
We briefly recall the definitions of log-concavity and related properties, as well as some key results on the preservation of these properties under certain polynomial operations. 
\begin{definition}
    A sequence of nonnegative real numbers $a_0,\ldots, a_n$ is said to
\begin{itemize}
    \item be \textit{log-concave} (LC) if for $0<k<n$,
    \begin{equation}
    \label{eq:LC_ineq}
        a_k^2 \geq a_{k-1}a_{k+1}
    \end{equation}
    \item be \textit{ultra log-concave} (ULC) if for $0<k<n$,
\begin{equation}
\label{eq:ULC_binom_ineq}
    \frac{a_k^2}{{n \choose k}^2} \geq \frac{a_{k-1}}{{n\choose k-1}} \frac{a_{k+1}}{{n \choose k+1}}
\end{equation}
 Note that this condition (\ref{eq:ULC_binom_ineq}) is equivalent to 
\begin{equation}
\label{eq:ULC_factor}
    k(n-k)a_k^2 \geq (n-k+1)(k+1)a_{k-1}a_{k+1}.
\end{equation}
\item be \textit{unimodal} if there exists an index $k$ such that 
\begin{equation}
    a_0 \leq \ldots a_{k-1} \leq a_k \geq \ldots \geq a_n
\end{equation}
\item \textit{have no internal zeros} if there do not exist indices $0\leq i <j<k \leq n$ such that $a_i,a_k \neq 0$ but $a_j = 0$. 
\end{itemize}
\end{definition}
We note that sometimes in the literature, nonnegativity is included in the definition of log-concavity. For the purposes of this paper, we do not take nonnegativity of the $a_i$ to be a necessary condition for log-concavity but do state explicitly when we are considering sequences of nonnegative real numbers specifically. 

The above definitions are related to each other in the following ways. First, ultra log-concave is a stronger condition than log-concave. This follows from the observation that the binomial coefficient sequence ${n\choose 0},\ldots, {n \choose n}$ is log-concave. 
Log-concave sequences are not necessarily unimodal nor have no internal zeros; the sequence $1,0,0,1$ is one such example. Positive log-concave sequences are, however, unimodal. Given a sequence of positive real numbers $a_0,\ldots, a_n$ that is log-concave, if there is an index $k$ such that $a_{k-1}\geq a_k$, then by log-concavity, $a_k \cdot \frac{a_k}{a_{k-1}} \geq a_{k+1}$, hence $a_k \geq a_{k+1}$. 

In \cite{Liggett_ULC_NegDependence_ConvoProd}, Liggett defines a variation of ultra log-concavity introduced by Pemantle. A sequence of real numbers $a_0,\ldots, a_n$ is said to be \textit{ultra log-concave of order $m$}, shorthand ULC($m$), if $a_k = 0$ for $k>m$ and 
\begin{equation}
\label{eq:ULC_binom_ineq_orderm}
    \frac{a_k^2}{{m \choose k}^2} \geq \frac{a_{k-1}}{{m\choose k-1}} \frac{a_{k+1}}{{m \choose k+1}}
\end{equation}
holds for $0<k<m$.
Once again, this inequality (\ref{eq:ULC_binom_ineq_orderm}) is equivalent to 
\begin{equation}
\label{eq:ULC_factor_orderm}
    k(m-k)a_k^2 \geq (m-k+1)(k+1)a_{k-1}a_{k+1}.
\end{equation}
The limiting case, ULC($\infty$), is given by the condition 
\begin{equation}
    ka_k^2 \geq (k+1)a_{k-1}a_{k+1}
\end{equation}
for all $0<k<n$. Note that if a sequence of real numbers if ULC($m$), then it is also ULC($m+1$) as $\frac{n+1}{n+2} > \frac{n}{n+1}$ for all $n \in \mathbb{N}$.

A polynomial with real coefficients is said to be log-concave, ultra log-concave (of order $m$), unimodal, or have no internal zeros if its coefficients form a sequence with the respective property. In \cite{Stanley_LogConcave_Survey}, Stanley presents a survey of examples of log-concave sequences, properties related to log-concavity, and ways to prove certain real polynomials are log-concave. One such method is given by a result from Newton \cite[Thm~2]{Stanley_LogConcave_Survey}: if $P(x)$ is a real polynomial with real zeros, then $P(x)$ is ultra log-concave. 

A key property of (ultra) log-concavity that we use in the following sections is the behavior of (ultra) log-concavity under taking products. More precisely, the statement is as follows.
\begin{proposition}
\label{prop:prod_log_concave}
Let $f(x), g(x)$ be two real polynomials whose coefficients are nonnegative with no internal zeros. Then $f(x)g(x)$ has no internal zeros and 
\begin{enumerate}
    \item if $f(x),g(x)$ are log-concave, then $f(x)g(x)$ is log-concave.
    \item if $f(x), g(x)$ are ultra log-concave, then $f(x)g(x)$ is ultra log-concave.
    \item  if $f(x)$ is ULC($n$) and $g(x)$ is ULC($m$), then $f(x)g(x)$ is ULC($n+m$). 
\end{enumerate}
\end{proposition}
The statement for log-concavity is proved in \cite[Prop~2]{Stanley_LogConcave_Survey} using linear algebra techniques. A more elementary proof is given in \cite[Thm~1]{WoongKook_ProductLogConcavePoly}. Liggett proves the statements for ultra log-concavity \cite[Thm~2]{Liggett_ULC_NegDependence_ConvoProd} as a statement for the convolution product of two ultra log-concave sequences that have no internal zeros. 

Both log-concavity and ultra log-concavity are also preserved under a shift of the dependent variable.
\begin{proposition}
\label{prop:LC_and_ULC_xPLUSn}
Let $f(x)$ be a real polynomial with nonnegative coefficients and no internal zeros. Then,
\begin{itemize}
    \item if $f(x)$ is log-concave, then $f(x+k)$ is log-concave for any positive integer $k$.
    \item if $f(x)$ is ultra-log-concave, then $f(x+k)$ is ultra log-concave for any positive integer $k$. 
\end{itemize}
\end{proposition}
See \cite[Thm~2]{Hoggar_Chromative_Poly_Log_Concavity} for the statement on log-concavity. The statement for ultra log-concavity follows from Example 2.26 and Theorem 3.4 of \cite{BrandenHuh_Lorentzian}. 

\section{Recursive Polynomials}
\label{sec:recursive_polynomials}
As partial progress towards proving Conjecture \ref{conj:log_concave_melonic}, we start by focusing on what we refer to as the four recursive polynomials involved in the recursive relation (\ref{eq:recursive_melonic}), namely $f_m(\mathbb{S}), g_m(\mathbb{S}), h_m(\mathbb{S}), \mathbb{B}_m$ for $m \geq 1$. Written in terms of the indeterminate $s$ (as a stand-in for $\mathbb{S}$), they are (for $m \geq 1$):
\begin{equation}
\label{eq:f_m(s)}
    f_m(s) := \frac{(s+1)^m-(-1)^m}{s+2} 
\end{equation}
\begin{equation}
\label{eq:g_m(s)}
    g_m(s) := m(s+1)^{m-1} - \frac{(s+1)^m-(-1)^m}{s+2}
\end{equation}
\begin{equation}
\label{eq:h_m(s)}
    h_m(s) := \frac{(s+1)^m + (-1)^m(s+1)}{s+2} 
\end{equation}
\begin{equation}
\label{eq:b_m(s)}
    b_m(s) := (s+1)\frac{(s+1)^m-(-1)^m}{s+2} + m(s+1)^{m-1} 
\end{equation}
Following equation (5.16) of \cite{Deletion_Contraction}, we also define
\begin{equation}
    f_0(s) = 0 \quad\quad g_0(s) = 0 \quad\quad h_0(s) = 1 \quad\quad b_0(s) = 0
\end{equation}
We may then make the following observations to rewrite the polynomials in forms that are easier to work with. First,
\begin{equation}
\label{eq:f_m_expansion}
    f_m(s) = \frac{(s+1)^m - (-1)^m}{(s+1) - (-1)} = \sum_{k=0}^{m-1} (s+1)^k(-1)^{m-1-k}. 
\end{equation}
The coefficients of $f_m(s)$ are given in Claim 3.5 of \cite{MR4620363}. Both $g_m, b_m$ may be written in terms of $f_m$, given below:
\begin{equation}
\label{eq:g_m(s)_in_f_m(s)}
    g_m(s) = m(s+1)^{m-1} - f_m(s)
\end{equation}
\begin{equation}
\label{eq:b_m(s)_in_f_m(s)}
    b_m(s) = m(s+1)^{m-1} + (s+1)f_m
\end{equation}
Finally, for $m\geq 1$, $h_m$ may be written in terms of $f_{m-1}$, as follows:
\begin{equation}
    \label{eq:h_m(s)_in_f_m(s)}
    h_m(s) = (s+1)f_{m-1}(s).
\end{equation}

\subsection{Log-concavity of the recursion polynomials}
The goal of this subsection is to prove Theorem \ref{thm:intro_recursive_LC}. In what follows, we say a polynomial  $p(s) = a_ns^n + \ldots + a_0 \in \mathbb{Z}[s]$ of degree $n$ has positive coefficients if $a_i>0$ for $0\leq i\leq n$. Note that this is equivalent to the condition that all coefficients are nonnegative, the constant term is positive, and the polynomial has no internal zeroes. 

To make use of Proposition \ref{prop:prod_log_concave}, we first derive a recursive formula for $f_m$. For all $m \geq 0$, we claim that 
\begin{equation}
\label{eq:f_m+1_from_f_m}
    f_{m+1}(s) = (s+1)f_m + (-1)^m.
\end{equation}
Observe that 
\begin{align*}
    f_{m+1}(s) &= \sum_{k=0}^m (s+1)^k(-1)^{m-k}
    = \sum_{k=1}^m (s+1)^k(-1)^{m-k} + (-1)^m\\
    &= (s+1)\sum_{k=0}^{m-1} (s+1)^k(-1)^{m-1-k} + (-1)^m = (s+1)f_m + (-1)^m,
\end{align*}
so indeed (\ref{eq:f_m+1_from_f_m}) holds. 

\begin{lemma}
\label{lem:f_m_LC}
For all $m\geq 1$, $f_m$ is a degree $m-1$ log-concave polynomial with nonnegative integral coefficients and no internal zeros. More precisely, for $m\geq 2$, all coefficients of degree $1\leq k\leq m-1$ are strictly positive. 
\end{lemma}
\begin{proof}
The statement on the degree of $f_m$ is clear from (\ref{eq:f_m_expansion}).
From (\ref{eq:f_m+1_from_f_m}), it is easy to see that the constant term is either $0$ or $1$ depending on the parity of $f$; see also Section \ref{sec:f_m_coeff}. Similarly, having established nonnegativity of the constant term, by induction using (\ref{eq:f_m+1_from_f_m}), the nonconstant coefficients are positive. Note that the positivity condition on the coefficients implies no internal zeroes.

We prove log-concavity by induction on $m$. Key to this will be the recursive formula (\ref{eq:f_m+1_from_f_m}) together with Proposition \ref{prop:prod_log_concave}. The first couple cases $f_1 = 1, f_2 = s$ are all trivially log-concave as monomials.  Now, let $m\geq 3$ and assume that $f_{m-1}$ is log-concave with positive nonconstant coefficients.  Then, by Proposition \ref{prop:prod_log_concave}, $(s+1)f_{m-1}$ is also log-concave with positive nonconstant coefficients. Let $a_k$ denote the degree $k$ coefficient of $f_m$, so $a_k>0$ for $1\leq k\leq m-1$ and $a_k^2 \geq a_{k+1}a_{k-1}$ for $2\leq k\leq m-1$. It thus suffices to check the log-concavity condition for the degree $1$ coefficient of $f_{m}$.
The coefficients of $f_m$ up through degree 4 are computed in terms of $m$ in Section \ref{sec:f_m_coeff}. 
Denoting by $a_k$ the degree $k$ coefficient, we check that 
\begin{equation*}
a_1^2 - a_0a_2 =
    \begin{cases}
     \frac{(m-1)^2}{4} - \frac{(m-1)^2}{4} = 0 & m \text{ is odd}  \\
     \frac{m^2}{4} > 0 & m \text{ is even}
    \end{cases}.
\end{equation*}
Thus, $f_m$ is log-concave.   
\end{proof}

\begin{corollary}
\label{cor:h_m_LC}
For all $m \geq 1$, $h_m$ is a degree $m-1$ log-concave polynomial with nonnegative integer coefficients and no internal zeros.     
\end{corollary}
\begin{proof}
As $f_m$, $(s+1)$ are log-concave with nonnegative integer coefficients and no internal zeros for all $m$, by Proposition \ref{prop:prod_log_concave}, (\ref{eq:h_m(s)_in_f_m(s)}) tells us that $h_m$ is also log-concave with nonnegative integer coefficients and no internal zeros for all $m \geq 1$. The statement on degree follows from Lemma \ref{lem:f_m_LC} and (\ref{eq:h_m(s)_in_f_m(s)}). 
\end{proof}

\begin{lemma}
\label{lem:g_m_LC}
For all $m \geq 1$, $g_m$ is a degree $m-1$ log-concave polynomial with positive integer coefficients.   
\end{lemma}
\begin{proof}
We adapt the method used to prove Lemma \ref{lem:f_m_LC}. It is clear that $g_1=0$ satisfies the statement, so it suffices to consider $m \geq 2$.  
First, for all $n\geq 1,m\geq 1$, we define
\begin{equation}
\label{eq:g_m_n_def}
    g_{m,n}(s) = n(s+1)^{m-1}- f_m.
\end{equation}
Observe that then
\begin{equation}
\label{eq:g_m_n_recursive}
    g_{m+1,n}(s) = n(s+1)^m - f_{m+1} = (s+1)g_{m,n} - (-1)^m.
\end{equation}
Note that for $n\geq 2$,
\begin{equation*}
    g_{1,n} = n - f_1 = n-1 \quad\quad g_{2,n} = n(s+1) - f_2 = (n-1)s + n
\end{equation*}
are both log-concave with positive integer coefficients. We fix $n \geq 2$ and proceed by induction on $m$ to show that $g_{m,n}$ is a log-concave polynomial of degree $m-1$ with positive integer coefficients for all $m \geq 2$. The case $m=2$ is given above. Let $m\geq 3$. Assume $g_{m-1,n}$ is a log-concave polynomial of degree $m-2$ with positive integer coefficients. The product $(s+1)g_{m-1,n}$ is then degree $m-1$ with positive integer coefficients which are log-concave by Proposition \ref{prop:prod_log_concave}. From (\ref{eq:g_m(s)_in_f_m(s)}), the constant term of $g_{m,n}$ is at least $n-1>0$. It then suffices to check that the degree $1$ coefficient of $g_{m,n}$ satisfies the log-concave inequality, which is done in Section \ref{sec:g_m_n_coeff}. 

Taking $n=m$ gives the statement for $g_m = g_{m,m}$ for $m\geq 2$. 
\end{proof}

\begin{lemma}
\label{lem:b_m_LC}
For all $m \geq 1$, $b_m$ is a degree $m$ log-concave polynomial with positive integer coefficients. 
\end{lemma}
\begin{proof}
We take a similar approach as for $g_m$. 

For integers $n, m \geq 1$, we define
\begin{equation}
\label{eq:b_m_n_def}
    b_{m,n}(s) = n(s+1)^{m-1} + (s+1)f_m
\end{equation}
Observe that then
\begin{equation}
\label{eq:b_m_n_recursive}
    b_{m+1, n}(s) = n(s+1)^m + (s+1)f_{m+1} = (s+1)b_{m,n} + (-1)^m(s+1)
\end{equation}
Note that 
\begin{equation*}
    b_{1,n} = n + (s+1)f_1 = s+ n+1 \quad\quad b_{2,n} = n(s+1) + (s+1)f_2 = s^2 + (n+1)s +n
\end{equation*}
are both log-concave with positive integer coefficients. In fact, as $f_m$ is degree $m-1$ with nonnegative integer coefficients by Lemma \ref{lem:f_m_LC}, $b_{m,n}$ is degree $m$ with positive integer coefficients for all $n,m\geq 1$.  

We prove by induction on $m$ that $b_{m,n}$ is log-concave for $m,n \geq 1$. Let $n\geq 1, m\geq 3$ and assume that $b_{m-1,n}$ is log-concave. Then, $(s+1)b_{m-1,n}$ is also log-concave, so by (\ref{eq:b_m_n_recursive}), to show that $b_{m,n}$ is log-concave, it suffices to check that the degree 2 and degree 1 coefficients both satisfy the log-concave inequality. We do this by explicit computations, given in Section \ref{sec:b_m_n_coeff}. 
We thus see that $b_{m,n}$ is log-concave, completing the inductive proof. Taking $n=m$ gives the desired result for $b_m = b_{m,m}$, $m\geq 1$. 
\end{proof}

Altogether, we have the following result. 
\begin{theorem}
\label{thm:f_g_h_b_LC}
For all $m\geq 0$, $f_m, g_m, h_m,$ and $b_m$ are log-concave with nonnegative integer coefficients and no internal zeros.     
\end{theorem}
\begin{remark}
Note that the result for banana graphs also establishes the same result for any string of bananas by the multiplicative property of the Grothendieck class of a graph and Proposition \ref{prop:prod_log_concave}.  
\end{remark}
For completion, we reprove Corollary 3.4 of \cite{MR4620363}. 
\begin{corollary}
\label{cor:melonic_no_int_zeros}
Let $T$ be a melonic construction. Suppose the graph obtained from $T$ has $n$ edges. We may then write $\mathbb{U}(T) = P(\mathbb{S})$ for some  degree $n$ polynomial $P(t)\in \mathbb{Z}[t]$ with positive coefficients. 
\end{corollary}
\begin{proof}


We first use the recursion described at the end of Section \ref{sec:melonic_graphs} to prove that $\mathbb{U}(T)$ may be written as a polynomial of degree $n$ in $\mathbb{S}$ with nonnegative integral coefficients. Recall that $f_m, g_m, h_m$ are all of degree $m-1$ while $b_m$ is of degree $m$. 
In particular, an $m$-banana has a Grothendieck class of degree $m$ in $\mathbb{S}$. The Grothendieck class of a string of bananas is just the product of the classes of the individual bananas in the string, hence the degree of the class in $\mathbb{S}$ is the total number of edges. If $T$ is a melonic construction whose last stage is $((1,\ldots, 1), p, k)$, then $\mathbb{U}(T) = (\mathbb{S}+2)^{r-1}\mathbb{U}(T')$ where $r$ is the length of the tuple $(1,\ldots, 1)$ and $T'$ is obtained from $T$ by omitting the last stage. Assuming that $\mathbb{U}(T')$ is of degree $d$ where $d$ is the number of edges of the graph determined by $T'$, we then have that $G$ has $d+ r-1$ edges and $\mathbb{U}(T)$ is of degree $d + r-1$. 

For the final case involving the deletion-contraction relation, use notation as in the end of Section \ref{sec:melonic_graphs} and let $d$ be the number of edges of $G'$. Then $G'/e, G' \setminus e$ both have $d-1$ edges and $G$ has $d + a_m -1$ edges. Let $G'''$ be the graph determined by $T'''$ and let $n'$ be the total number of edges in the (one or two) strings of bananas whose attachment to vertices in $G'''$ give $G' \setminus e$. Then $G'''$ has $d-1 - n'$ edges. Assume the Grothendieck classes of $T', T'',T'''$ are of degrees equal to the number of edges of $G', G'', G'''$ respectively. By (\ref{eq:recursive_melonic}), we then have that $\mathbb{U}(T)$ is a polynomial given by a sum of polynomials with nonnegative coefficients of degree $a_m-1+ d, a_m-1 + d-1, a_m-1+ n'+d-1-n'$ respectively, hence $\mathbb{U}(T)$ is of degree $d + a_m-1$, the number of edges of $G$. 

We also use the recursion to prove the statement on positivity of the coefficients. First, observe that $b_m$ has positive coefficients by Lemma \ref{lem:b_m_LC}, hence any product $\prod_{i=1}^r b_{m_i}$ will also have positive coefficients. This covers the case of a string of bananas. For a melonic graph $G$ obtained from another $G'$ by splitting an edge $r$ times, assuming that $\mathbb{U}(G')$ has positive coefficients, $\mathbb{U}(G) = (\mathbb{S}+2)^{r-1}\mathbb{U}(G')$ also has positive coefficients as this property is preserved under products. 

Finally, we consider the case where the last stage of the construction $T$ is of the form $t_l = ((a_1,\ldots, a_r), p, k )$ with $a_m = \max\{a_i\}_{1\leq i\leq r} > 1$. 
Observe that while $f_{a_m}, h_{a_m}$ may have a zero constant term depending on the parity of $a_m$, $g_{a_m}$ has a nonzero constant term for ${a_m}>1$. Assuming $\mathbb{U}(T'), \mathbb{U}(T''),\mathbb{U}(T''')$ all have positive coefficients, it then follows that $g_{a_m}\mathbb{U}(T'')$ has positive coefficients and $f_{a_m}\mathbb{U}(T'),$ $h_{a_m}\mathbb{U}(T''')\prod_{i\neq m}b_{a_i}$ both have positive nonconstant coefficients. Thus, $\mathbb{U}(T)$ has positive coefficients.
\end{proof}

\subsection{Ultra log-concavity for the recursion polynomials}
A natural question to ask following Theorem \ref{thm:f_g_h_b_LC} is if the recursive polynomials satisfy the stronger property of being ultra log-concave. 

We first fix some language. Let $p(s) \in \mathbb{R}[s]$ be an arbitrary polynomial of degree $d$. Denote by $a_j$ the coefficient of the degree $j$ term of $p$ for $0\leq j \leq d$. We set $a_j=0$ for integers $j>d$ and $j<0$. We will say that the coefficient of degree $k>0$ satisfies the ULC inequality (\ref{eq:ULC_binom_ineq}, \ref{eq:ULC_factor}) if $k(d-k)a_k^2 \geq (k+1)(d-k+1)a_{k-1}a_{k+1}$.

We give the results of checking ultra log-concavity for $m=1$ through $m=10$ below. The computations were carried out in Sage \cite{Sage}. 
For the coefficients, the list has indexes corresponding to the degree of the term, i.e. a coefficient list $[1,2,3]$ would correspond to the polynomial $3s^2 + 2s+1$. 
\begin{table}[H]
    \centering
    \begin{tabular}{|c|c|c|c|}
    \hline 
        $m$ &  Coefficients & ULC & Degrees that fail ULC\\
        \hline 
        1 & [1] &  True & None\\
        2 & [0, 1] & True & None \\
        3 & [1, 1, 1] & False & 1\\
        4 & [0, 2, 2, 1] & False & 2\\
        5 & [1, 2, 4, 3, 1] & False & 1, 3\\
        6 & [0, 3, 6, 7, 4, 1] & False & 2, 4\\ 
        7 & [1, 3, 9, 13, 11, 5, 1] & False & 1, 3, 4, 5\\
        8 & [0, 4, 12, 22, 24, 16, 6, 1] & False & 2, 4, 5, 6\\
        9 & [1, 4, 16, 34, 46, 40, 22, 7, 1] & False & 1, 3, 4, 5, 6, 7\\
        10 & [0, 5, 20, 50, 80, 86, 62, 29, 8, 1] & False & 2, 4, 5, 6, 7, 8 \\
        \hline 
    \end{tabular}
    \caption{Checking ultra log-concavity for $f_m$, $m$=1-10}
    \label{tab:f_m_ULC_1-10}
\end{table}

\begin{table}[H]
    \centering
    \begin{tabular}{|c|c|c|c|}
    \hline 
        $m$ &  Coefficients & ULC & Degrees that fail ULC\\ 
        \hline 
        1 & [0] &  True & None\\ 
        2 & [2, 1] & True & None \\ 
        3 & [2, 5, 2] & True & None \\ 
        4 & [4, 10, 10, 3] & False & 1\\ 
        5 & [4, 18, 26, 17, 4] & False & 2\\ 
        6 & [6, 27, 54, 53, 26, 5] & False & 1\\ 
        7 & [6, 39, 96, 127, 94, 37, 6] & False & 2\\ 
        8 & [8, 52, 156, 258, 256, 152, 50, 7] & False & 1\\
        9 & [8, 68, 236, 470, 584, 464, 230, 65, 8] & False & 2\\
        10 & [10, 85, 340, 790, 1180, 1174, 778, 331, 82, 9] & False & 1\\
        \hline 
    \end{tabular}
    \caption{Checking ultra log-concavity for $g_m$, $m$=1-10}
    \label{tab:g_m_ULC_1-10}
\end{table}

\begin{table}[H]
    \centering
    \begin{tabular}{|c|c|c|c|}
    \hline 
        $m$ &  Coefficients & ULC & Degrees that fail ULC\\
        \hline 
        1 & [0] &  True & None\\
        2 & [1, 1] & True & None \\
        3 & [0, 1, 1] & True & None\\
        4 & [1, 2, 2, 1] & False & 1, 2\\
        5 & [0, 2, 4, 3, 1] & False & 3\\
        6 & [1, 3, 6, 7, 4, 1] & False & 1, 2, 4\\ 
        7 & [0, 3, 9, 13, 11, 5, 1] & False & 3, 4, 5\\
        8 & [1, 4, 12, 22, 24, 16, 6, 1] & False & 1, 2, 4, 5, 6\\
        9 & [0, 4, 16, 34, 46, 40, 22, 7, 1] & False & 3, 4, 5, 6, 7\\
        10 & [1, 5, 20, 50, 80, 86, 62, 29, 8, 1] & False & 1, 2, 4, 5, 6, 7, 8\\
        \hline 
    \end{tabular}
    \caption{Checking ultra log-concavity for $h_m$, $m$=1-10}
    \label{tab:h_m_ULC_1-10}
\end{table}

\begin{table}[H]
    \centering
    \begin{tabular}{|c|c|c|c|}
    \hline 
        $m$ &  Coefficients & ULC & Degrees that fail ULC\\
        \hline 
        1 & [2, 1] &  True & None\\
        2 & [2, 3, 1] & True & None \\
        3 & [4, 8, 5, 1] & True & None \\
        4 & [4, 14, 16, 7, 1] & True & None \\
        5 & [6, 23, 36, 27, 9, 1] & False & 1 \\
        6 & [6, 33, 69, 73, 41, 11, 1] & True & None \\ 
        7 & [8, 46, 117, 162, 129, 58, 13, 1] & False & 1\\
        8 & [8, 60, 184, 314, 326, 208, 78, 15, 1] & True & None\\
        9 & [10, 77, 272, 554, 710, 590, 314, 101, 17, 1] & False & 1\\
        10 & [10, 95, 385, 910, 1390, 1426, 988, 451, 127, 19, 1] & True & None\\
        \hline 
    \end{tabular}
    \caption{Checking ultra log-concavity for $b_m$, $m$=1-10}
    \label{tab:b_m_ULC_1-10}
\end{table}

\begin{lemma}
\label{lem:g_m_n_ULC_3}
Taking $g_{m,n}$ as defined previously (\ref{eq:g_m_n_def}), for all $m, n \geq 1$, all coefficients of degree at least $3$ satisfy the ULC inequality.     
\end{lemma}
\begin{proof}
Fix $n\geq 1$. We denote by $g_{m,n,2}$ the polynomial obtained from $g_{m,n}$ by omitting the terms of degree 0 and 1. In other words, if $a_0, a_1$ are the coefficients of the terms of degree 0,1 in $g_{m,n}$ respectively, then 
\begin{equation*}
    g_{m,n}(s) = g_{m,n,2}(s) + a_1s + a_0.
\end{equation*}
 The condition that all coefficients of degree at least $3$ satisfy the ULC inequality is equivalent to $g_{m,n,2}$ being ultra log-concave. This is trivially true when $g_{m,n}$ is of degree less than $4$, i.e. when $m \leq 4$. We proceed by induction. Let $m\geq 5$ and assume that $g_{m-1, n,2}$ is ultra log-concave. By Proposition \ref{prop:prod_log_concave}, $(s+1)g_{m-1,n,2}$ is then ultra log-concave. 
 
 Let $a_0,a_1$ denote the coefficients of the degree 0,1 terms of $g_{m-1,n}$ respectively.  Observe using (\ref{eq:g_m_n_recursive}) that then
 \begin{align*}
     g_{m,n}(s)& = (s+1)g_{m-1,n,2} + (s+1)(a_1s+a_0) - (-1)^{m-1}\\
     &= (s+1)g_{m-1, n,2} + a_1s^2 + (a_1+a_0)s + a_0 - (-1)^{m-1},
 \end{align*}
 so
 \begin{equation*}
     g_{m,n,2}(s) =  (s+1)g_{m-1,n,2} + a_1s^2. 
 \end{equation*}
 It then suffices to check that the degree 3 coefficient of $g_{m,n}$ satisfies the ULC inequality. This is done in Section \ref{sec:g_m_n_coeff}. 
\end{proof}

\begin{lemma}
\label{lem:b_m_n_ULC_3}
Taking $b_{m,n}$ as defined previously (\ref{eq:b_m_n_def}), for all $m \geq 1, n \geq 3$, all coefficients of degree at least $3$ satisfy the ULC inequality.     
\end{lemma}
\begin{proof}
Fix $n\geq 2$. As in Lemma \ref{lem:g_m_n_ULC_3}, we denote by $b_{m,n,2}$ the polynomial obtained from $b_{m,n}$ by omitting the terms of degree 0 and 1.
The condition that all coefficients of degree at least $3$ satisfy the ULC inequality is then equivalent to $b_{m,n,2}$ being ultra log-concave. This is trivially true when $b_{m,n}$ is of degree less than $4$, i.e. when $m \leq 3$. We proceed by induction. Let $m\geq 4$ and assume that $b_{m-1, n,2}$ is ultra log-concave. By Proposition \ref{prop:prod_log_concave}, $(s+1)b_{m-1,n,2}$ is then ultra log-concave.  
Let $a_0,a_1$ denote the coefficients of the degree 0,1 terms of $b_{m-1,n}$ respectively.  Observe using (\ref{eq:b_m_n_recursive}) that then
 \begin{align*}
     b_{m,n}(s)& = (s+1)b_{m-1,n,2} + (s+1)(a_1s+a_0) + (-1)^{m-1}(s+1)\\
     &= (s+1)b_{m-1, n,2} + a_1s^2 + (a_1+a_0 +(-1)^{m-1})s + a_0 + (-1)^{m-1},
 \end{align*}
 so
 \begin{equation*}
     b_{m,n,2}(s) =  (s+1)b_{m-1,n,2} + a_1s^2. 
 \end{equation*}
 It then suffices to check that the degree 3 coefficient of $b_{m,n}$ satisfies the ULC inequality. This is done in Section \ref{sec:b_m_n_coeff}. Note that the assumption that $n\geq 3$ rather than $n\geq 1$ or $n\geq 2$ is necessary here as the statement is false for $m=6, n=1$ and $m=4, n=2$.  
\end{proof}

\begin{proposition}
\label{prop:f_g_h_b_ULC}
(ultra log-concavity of $f_m, g_m, h_m, b_m$)

For $m\geq 4$, $f_m, g_m, h_m,$ and $b_m$ have the following properties. 

When $m$ is odd:
\begin{itemize}
    \item $f_m$ is not ultra log-concave and the coefficients of degree 1 and 3 fail the ULC inequality while the coefficient of degree 2 satisfies the ULC inequality.
    \item $g_m$ is not ultra log-concave and the only coefficient that fails the ULC inequality is degree 2.
    \item $h_m$ is not ultra log-concave and the coefficient of degree 3 fails the ULC inequality while the coefficient of degrees 1 and 2 satisfy the ULC inequality.
    \item $b_m$ is not ultra log-concave and the only coefficient that fails the ULC inequality is degree 1.
\end{itemize}

When $m$ is even:
\begin{itemize}
    \item $f_m$ is not ultra log-concave and the coefficient of degree 2 fails the ULC inequality while the coefficients of degrees 1 and 3 satisfy the ULC inequality.
    \item $g_m$ is not ultra log-concave and the only coefficient that fails the ULC inequality is degree 1.
    \item  $h_m$ is not ultra log-concave and the coeficients of degrees 1 and 2 fail the ULC inequality while the coefficient of degree 2 satisfies the ULC inequality
    \item $b_m$ is ultra log-concave 
\end{itemize}
\end{proposition}
\begin{proof}
The statements for $g_m,b_m$ are given by taking $n=m$ in Lemmas \ref{lem:g_m_n_ULC_3},\ref{lem:b_m_n_ULC_3} and checking the ULC inequality for the coefficients of degrees 1 and 2, which is done in Sections \ref{sec:g_m_n_coeff} and \ref{sec:b_m_n_coeff}. The statements for $f_m$ may be checked by explicit computation, as in Section \ref{sec:f_m_coeff}. For $h_m$, observe using equations (\ref{eq:h_m(s)_in_f_m(s)}) and (\ref{eq:f_m+1_from_f_m}) that 
\begin{equation}
\label{eq:h_m_is_f_m_off1}
    h_m(s) = f_m(s) + (-1)^m.
\end{equation}
Then, the only coefficient of $h_m$ that differs from $f_m$ is the degree 0 coefficient. It then suffices to check whether or not the degree 1 coefficient satisfies the ULC inequality. Let $a_j$ denote the coefficient of the degree $j$ term of $h_m$. 
For $m$ odd, $a_0 = 0$, so the ULC inequality is satisfied. For $m \geq 4$ even, $a_0 = 1$, so using the coefficients of $f_m$ calculated in Section \ref{sec:f_m_coeff}, we have
\begin{equation*}
    (m-2)a_1^2 - 2(m-1)a_0a_2
    = \frac{m(m-2)}{2} - \frac{m(m-2)(m-1)}{2}
    = -\frac{m(m-2)^2}{2} < 0,
\end{equation*}
hence the degree 1 coefficient fails the ULC inequality. 
\end{proof}

Computer computations for $m=1$ through $m=100$ suggest the following. 
\begin{conjecture}
For $m\geq 4$, both $f_m$ and $h_m$ are not ultra log-concave and all coefficients of degree $4 \leq j \leq m-2$ fail the ULC inequality. 
\end{conjecture}

\subsection{Ultra-log-concavity of order $m$ for the recursion polynomials}
The polynomials $f_m, g_m, h_m$ are slightly better behaved with respect to being ultra log-concave of order $m$ (rather than $m-1$). Note that $b_m$ is of degree $m$, hence the results for the ultra log-concavity of order $m$ for $b_m$ are covered by Proposition \ref{prop:f_g_h_b_ULC}, though we include them again in Proposition \ref{prop:f_g_h_b_ULC_m} for completeness.  As before, let $p(s) \in \mathbb{R}[s]$ be a polynomial of degree $d<m$ and denote by $a_j$ the coefficient of the degree $j$ term of $p$ for $0 \leq j \leq d$ with $a_j=0$ for $j > d$. 
We say that the coefficient of degree $k>0$ satisfies the ULC($m$) inequality (\ref{eq:ULC_factor_orderm}) if $k(m-k)a_k^2 \geq (k+1)(m-k+1)a_{k-1}a_{k+1}$.

We first give the results of ULC($m$) for $f_m, g_m, h_m$ computed for $m=1$ through $m=10$ below. As before, the coefficients are listed in order $[a_0,\ldots, a_{m-1}]$. 
\begin{table}[H]
    \centering
    \begin{tabular}{|c|c|c|c|}
    \hline 
        $m$ &  Coefficients & ULC($m$) & Degrees that fail ULC($m$)\\
        \hline 
        1 & [1] &  True & None\\
        2 & [0, 1] & True & None \\
        3 & [1, 1, 1] & False & 1\\
        4 & [0, 2, 2, 1] & False & 2\\
        5 & [1, 2, 4, 3, 1] & False & 1\\
        6 & [0, 3, 6, 7, 4, 1] & False & 2\\ 
        7 & [1, 3, 9, 13, 11, 5, 1] & False & 1\\
        8 & [0, 4, 12, 22, 24, 16, 6, 1] & False & 2\\
        9 & [1, 4, 16, 34, 46, 40, 22, 7, 1] & False & 1\\
        10 & [0, 5, 20, 50, 80, 86, 62, 29, 8, 1] & False & 2 \\
        \hline 
    \end{tabular}
    \caption{Checking ULC($m$) for $f_m$, $m$=1-10}
    \label{tab:f_m_ULC_m_1-10}
\end{table}

\begin{table}[H]
    \centering
    \begin{tabular}{|c|c|c|c|}
    \hline 
        $m$ &  Coefficients & ULC($m$) & Degrees that fail ULC($m$)\\ 
        \hline 
        1 & [0] &  True & None\\ 
        2 & [2, 1] & True & None \\ 
        3 & [2, 5, 2] & True & None \\ 
        4 & [4, 10, 10, 3] & True & None\\ 
        5 & [4, 18, 26, 17, 4] & True & None\\ 
        6 & [6, 27, 54, 53, 26, 5] & False & 1\\ 
        7 & [6, 39, 96, 127, 94, 37, 6] & True & None\\ 
        8 & [8, 52, 156, 258, 256, 152, 50, 7] & False & 1\\
        9 & [8, 68, 236, 470, 584, 464, 230, 65, 8] & True & None\\
        10 & [10, 85, 340, 790, 1180, 1174, 778, 331, 82, 9] & False & 1\\
        \hline 
    \end{tabular}
    \caption{Checking ULC($m$) for $g_m$, $m$=1-10}
    \label{tab:g_m_ULC_m_1-10}
\end{table}

\begin{table}[H]
    \centering
    \begin{tabular}{|c|c|c|c|}
    \hline 
        $m$ &  Coefficients & ULC($m$) & Degrees that fail ULC($m$)\\
        \hline 
        1 & [0] &  True & None\\
        2 & [1, 1] & True & None \\
        3 & [0, 1, 1] & True & None\\
        4 & [1, 2, 2, 1] & False & 1, 2\\
        5 & [0, 2, 4, 3, 1] & True & None\\
        6 & [1, 3, 6, 7, 4, 1] & False & 1, 2\\ 
        7 & [0, 3, 9, 13, 11, 5, 1] & True & None\\
        8 & [1, 4, 12, 22, 24, 16, 6, 1] & False & 1, 2\\
        9 & [0, 4, 16, 34, 46, 40, 22, 7, 1] & True & None\\
        10 & [1, 5, 20, 50, 80, 86, 62, 29, 8, 1] & False & 1, 2\\
        \hline 
    \end{tabular}
    \caption{Checking ULC($m$) for $h_m$, $m$=1-10}
    \label{tab:h_m_ULC_m_1-10}
\end{table}

\begin{proposition}
\label{prop:f_g_h_b_ULC_m}
(ultra log-concavity of order $m$ of $f_m, g_m, h_m, b_m$)

For $m\geq 6$, $f_m, g_m, h_m,$ and $b_m$ have the following properties. 

When $m$ is odd:
\begin{itemize}
    \item $f_m$ is not ULC($m$) and the only coefficient that fails the ULC($m$) inequality is degree $1$.
    \item $g_m$ is ULC($m$).
    \item $h_m$ is ULC($m$).
    \item $b_m$ is not ULC($m$) and the only coefficient that fails the ULC($m$) inequality is degree 1.
\end{itemize}

When $m$ is even:
\begin{itemize}
    \item $f_m$ is not ULC($m$) and the only coefficient that fails the ULC($m$) inequality is degree $2$.
    \item $g_m$ is not ULC($m$) and the only coefficient that fails the ULC($m$) inequality is degree $1$.
    \item  $h_m$ is not ULC($m$) and the only coefficients that fail the ULC($m$) inequality are degrees $1,2$.
    \item $b_m$ is ULC($m$).
\end{itemize}
\end{proposition}
\begin{proof}
We first prove by induction on $m$ that for all $m\geq 0$, all coefficients of degree at least 3 satisfy the ULC($m$) inequality for $f_m$. As before, we denote by $f_{m,2}$ the polynomial obtained from $f_m$ by omitting the terms of degree 0 and 1. The condition that all coefficients of degree at least 3 satisfy the ULC($m$) inequality for $f_m$ is equivalent to $f_{m,2}$ being ULC($m$). Let $a_0,a_1$ be the coefficients of the terms of degree 0, 1 in $f_{m}$. Then, using (\ref{eq:f_m+1_from_f_m}),
\begin{equation*}
    f_{m+1}(s) 
    = (s+1)f_{m,2} + (s+1)(a_1s+a_0) + (-1)^m
    = (s+1)f_{m,2} + a_1s^2 + (a_0+a_1)s + a_0 + (-1)^m,
\end{equation*}
so 
\begin{equation*}
    f_{m+1,2} = (s+1)f_{m,2} + a_1s^2. 
\end{equation*}
For $0 \leq m \leq 4$, $f_{m}$ is of degree at most $3$, hence $f_{m,2}$ is trivially ULC($m$). Let $m\geq 5$ and assume that $f_{m-1,2}$ is ULC($m-1$). By Proposition \ref{prop:prod_log_concave}, $(s+1)f_{m,2}$ is then ULC($m$). It then suffices to check that the degree 3 coefficients of $f_m$ satisfies the ULC($m$). This is done explicitly in Section \ref{sec:f_m_coeff}. The checking of the statements about the degree 1 and 2 coefficients of $f_m$ may also be found in Section \ref{sec:f_m_coeff}. 

As before, by (\ref{eq:h_m_is_f_m_off1}), the statements for $h_m$ follow from that of $f_m$ and checking ULC($m$) inequality for the degree 1 coefficient of $h_m$, as below:
\begin{equation*}
    m \text{ odd}: \quad (m-1)a_1^2 - 2ma_0a_2 = \frac{(m-1)^3}{4} > 0
\end{equation*}
\begin{equation*}
    m \text{ even}: \quad (m-1)a_1^2 - 2ma_0a_2
    = \frac{m^2(m-1)}{4} - \frac{m^2(m-2)}{2}
    = - \frac{m^2(m-3)}{4} < 0
\end{equation*}

Finally, the statements for $g_m$ and $b_m$ follow from the previous proposition. More precisely, as ULC($m-1$) implies ULC($m$), by Proposition \ref{prop:f_g_h_b_ULC}, all coefficients of degree at least 3 of $g_m$ satisfy the ULC($m$) inequality. It then suffices to check the degree 1 and 2 coefficients, which is done in Section \ref{sec:g_m_n_coeff}. 
\end{proof}

\section{Necklace Graphs}
In this section, we take necklace graphs to be graphs obtained by multiplying the edges of a polygon.  Equivalently, a necklace graph is obtained from a string of bananas by identifying the two outer endpoints. This section is inspired by Remark 2.2 of \cite{MR4620363} where the authors introduce the necklace graph consisting of only $2$-bananas, which we discuss shortly. 
The main result of this section is log-concavity of the Grothendieck classes of a specific infinite subfamily of necklace graphs, which we call \textit{clasped necklaces}, described below. 

\subsection{Necklace graph of $2$-bananas}
In Remark 2.2 of \cite{MR4620363}, the authors introduce the necklace graph consisting of only $2$-bananas. A corresponding melonic construction for the necklace graph of $n$ $2$-bananas is $(((3), 0, 1), ((2,\ldots, 2), 1,1))$ where the tuple $(2,\ldots, 2)$ has length $n-1$. 
More generally, for $m\geq 1, n\geq 2$, we denote by $G_{m,n}$ the necklace graph of $n$ $m$-bananas with melonic construction given by $(((m+1), 0, 1), ((m,\ldots, m), 1,1))$ where $(m,\ldots, m)$ has length $n-1$. We take $G_{m, 1}$ to be the graph with one vertex and $m$-loops. Note that $G_{m,2}$ is just a $2m$-banana. Observe that for $n\geq 3$, $G_{1,n}$ is an $n$-sided polygon, hence may be obtained from a $2$-banana by splitting an edge $n-2$ times. Its Grothendieck class is then
\begin{equation}
    \mathbb{U}(G_{1,n}) = \mathbb{B}_2(\mathbb{S}+2)^{n-2}
\end{equation}
which is ultra log-concave in $\mathbb{S}$ by Proposition \ref{prop:prod_log_concave} as $\mathbb{B}_2$ is ultra log-concave in $\mathbb{S}$. 

The necklace graph of $n$ $2$-bananas, $G_{2, n}$, has  Grothendieck class
\begin{equation}
    \mathbb{U}(G_{2,n}) = (\mathbb{T}^n + n\mathbb{T}^{n-1}-1)(\mathbb{T}+1)^{n-1}\mathbb{T} = ((\mathbb{S}+1)^n+n(\mathbb{S}+1)^{n-1}-1)(\mathbb{S}+2)^{n-1}(\mathbb{S}+1)
\end{equation}
as given in \cite[Remark~6.3]{MR4620363}. This formula may be proved by induction using (\ref{eq:necklace_delectioncontraction}) and (\ref{eq:clasp_necklace_class_T})  below. By observing that 
\begin{equation*}(\mathbb{S}+1)^n+n(\mathbb{S}+1)^{n-1}-1
    = (\mathbb{S}+1)^{n-1}(\mathbb{S}+1+n) - 1,
\end{equation*}
we see that $\mathbb{U}(G_{2,n})$ is also ultra log-concave in $\mathbb{S}$.

\subsection{Clasped necklaces}
For $m \geq 1, n\geq 2$, define the \textit{clasped necklace} $G_{m,n}'$ to be the necklace graph of $n-1$ $m$-bananas and a single $1$-banana. The clasped necklace $G'_{3,7}$ is pictured in Figure \ref{fig:clasped_necklace_3_7} below. 
The graph has melonic construction $((m+1, 0,1), ((1, m,\ldots, m), 1, 1))$ where $(1, m,\ldots, m)$ has length $n-1$. Denoting the single $1$-banana by $e$, we see that $G_{m,n} = (G_{m,n}')_{me}$. Note that $G_{m, 2}'$ is a $(m+1)$-banana. 

\begin{figure}
    \centering
    \includegraphics[width=0.4\linewidth]{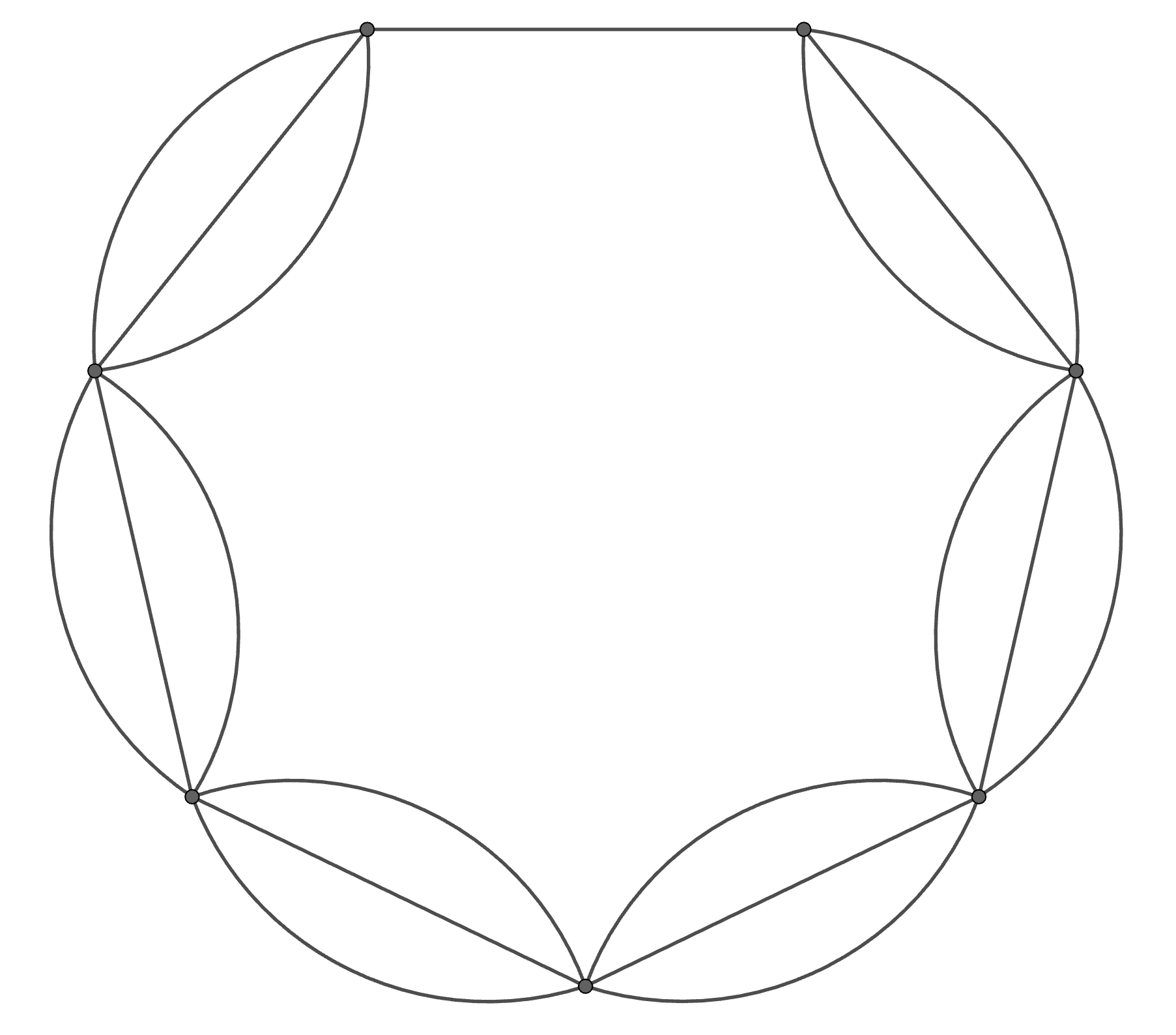}
    \caption{The clasped necklace $G'_{3,7}$ with $6$ $3$-bananas. }
    \label{fig:clasped_necklace_3_7}
\end{figure}

In what follows, we fix $m\geq 2$ and $n\geq 3$. For simplicity, we will also consider polynomials in terms of $\mathbb{T}$ instead of $\mathbb{S}$. 
Note that for $n\geq 3$,
 $G_{m,n}' / e = G_{m, n-1}$ and $G_{m,n}' \setminus e$ is a string of $n-1$ $m$-bananas. The deletion-contraction relation then gives 
 \begin{equation}
\label{eq:necklace_delectioncontraction}
     \mathbb{U}(G_{m,n}) = f_m\mathbb{U}(G_{m,n}') + g_m\mathbb{U}(G_{m, n-1}) + h_m\mathbb{B}_m^{n-1}.
 \end{equation}
Let $G_{m,n}''$ be the graph obtained from $G'_{m,n-1}$ by splitting the edge $e$, so $G_{m, n}''$ is the necklace graph of $2$ $1$-bananas and $n-2$ $m$-bananas. Then, $G_{m,n}' = (G_{m, n}'')_{me'}$ where $e'$ is an edge added from the splitting of $e$. Observe that then $G_{m,n}'' / e' = G_{m, n-1}'$ and $G_{m,n}'' \setminus e'$ is a string of $n-2$ $m$ bananas and a single $1$-banana. It follows that 
\begin{equation}
\begin{cases}
 \mathbb{U}(G_{m,n}'') = (\mathbb{T}+1)\mathbb{U}(G'_{m, n-1})\\
 \mathbb{U}(G''_{m, n} / e') = \mathbb{U}(G'_{m, n-1}) \\
 \mathbb{U}(G''_{m, n}\setminus e') = (\mathbb{T}+1)\mathbb{B}_m^{n-2}
\end{cases},
\end{equation}
hence
\begin{equation}
\label{eq:clasp_necklace_deletioncontraction}
    \mathbb{U}(G'_{m,n}) = f_m(\mathbb{T}+1)\mathbb{U}(G'_{m, n-1})
    + g_m\mathbb{U}(G'_{m, n-1})  + h_m(\mathbb{T}+1)\mathbb{B}_m^{n-2}. 
\end{equation}
We arrive at the following closed form for the Grothendieck class of $G'_{m,n}$. 
\begin{lemma}
\label{lem:clasp_necklace_class_T}
For $m,n \geq 2$, the necklace graph of $n-1$ $m$-bananas and a single $1$-banana has Grothendieck class 
\begin{equation}
\label{eq:clasp_necklace_class_T}
    \mathbb{U}(G_{m,n}') = \mathbb{T}(\mathbb{T}+1)\mathbb{B}_m^{n-2}\left(\mathbb{T}^{m-1} + \sum_{k=0}^{m-2}(-1)^{m-2-k}(n+k-1)\mathbb{T}^k \right).
\end{equation}
\end{lemma}
\begin{proof}
We first record some useful expressions involving the recursive polynomials in terms of $\mathbb{T}$. For $m\geq 2$,
\begin{equation}
\label{eq:B_m_T+1_factor}
    \mathbb{B}_m = \mathbb{T}(\mathbb{T}+1)\left(\mathbb{T}^{m-2} + \sum_{k=0}^{m-3}(-1)^{m-3-k}(k+1)\mathbb{T}^k \right).
\end{equation}
This may be proved by induction on $m$ using (\ref{eq:f_m+1_from_f_m}) written in terms of $\mathbb{T}$. Similarly, using (\ref{eq:g_m(s)_in_f_m(s)}), (\ref{eq:b_m(s)_in_f_m(s)}), we see that 
\begin{equation}
\label{eq:f_m(T+1)_plus_g_m_equals_B_m}
    f_m(\mathbb{T}+1) + g_m =  \mathbb{B}_m.
\end{equation}
Finally, using (\ref{eq:h_m(s)_in_f_m(s)}),
\begin{equation}
\label{eq:h_m_(T+1)}
    h_m(\mathbb{T}+1) = \mathbb{T}(\mathbb{T}+1)\sum_{k=0}^{m-2}(-1)^{m-2-k}\mathbb{T}^k.
\end{equation}

Fix $m\geq 2$. Then by (\ref{eq:B_m_T+1_factor}),
\begin{equation*}
    \mathbb{U}(G_{m,2}') = \mathbb{B}_{m+1} = \mathbb{T}(\mathbb{T}+1)\left(\mathbb{T}^{m-1} + \sum_{k=0}^{m-2}(-1)^{m-2-k}(k+1)\mathbb{T}^k \right),
\end{equation*}
so (\ref{eq:clasp_necklace_class_T}) is satisfied for $n=2$. We proceed by induction on $n$. Let $n\geq 2$ and assume (\ref{eq:clasp_necklace_class_T}) holds for $n$. 
By (\ref{eq:clasp_necklace_deletioncontraction}) together with (\ref{eq:f_m(T+1)_plus_g_m_equals_B_m}), (\ref{eq:h_m_(T+1)}), we then have
\begin{align*}
    \mathbb{U}(G'_{m, n+1})
    &= (f_m(\mathbb{T}+1) + g_m)\mathbb{U}(G'_{m,n}) + h_m(\mathbb{T}+1)\mathbb{B}_m^{n-1}\\
    &= \mathbb{T}(\mathbb{T}+1)\mathbb{B}_m^{n-1}\left(\mathbb{T}^{m-1} + \sum_{k=0}^{m-2}(-1)^{m-2-k}(n+k-1)\mathbb{T}^k  + \sum_{k=0}^{m-2}(-1)^{m-2-k}\mathbb{T}^k\right)\\
    &=  \mathbb{T}(\mathbb{T}+1)\mathbb{B}_m^{n-1}\left(\mathbb{T}^{m-1} + \sum_{k=0}^{m-2}(-1)^{m-2-k}(n+k)\mathbb{T}^k\right).
\end{align*}
\end{proof}

\begin{proposition}[Log-concavity for clasped necklaces]
\label{prop:clasp_necklace_LC}
For $m \geq 1, n\geq 2$, the class of the clasped necklace graph of $n-1$ $m$ bananas and a single $1$-banana is log-concave in $\mathbb{S}$. 
\end{proposition}
\begin{proof}
When $m=1$,  $G'_{1,n} = G_{1,n}$, hence its Grothendieck class is ultra log-concave in $\mathbb{S}$ as discussed above. For $m,n\geq 2$, the formula for the Grothendieck class (\ref{eq:clasp_necklace_class_T}) in terms of $s:= \mathbb{S}$ becomes
\begin{equation*}
    \mathbb{U}(G'_{m,n})
    = (s+1)(s+2)b_m^{n-2}\left((s+1)^{m-1} +\sum_{k=0}^{m-2}(-1)^{m-2-k}(n+k-1)(s+1)^k\right).
\end{equation*}
By Proposition \ref{prop:prod_log_concave} and Theorem \ref{thm:f_g_h_b_LC}, 
it suffices to show that 
\begin{equation*}
    p_{m,n}(s) := (s+1)^{m-1} +\sum_{k=0}^{m-2}(-1)^{m-2-k}(n+k-1)(s+1)^k
\end{equation*}
is log-concave with no internal zeros. For $m=2, n\geq 2$, we have 
\begin{equation*}
    p_{2, n}(s) = s+1 + n-1 = s+n,
\end{equation*}
which has no internal zeros and is trivially log-concave. 
Let $m\geq 3$ and assume that $p_{m-1,j}$ is log-concave with no internal zeros for all $j \geq 2$.
Observe that 
\begin{align*}
    p_{m,n}(s) &= (s+1)(s+1)^{m-2} + (s+1)\sum_{k=1}^{m-2}(-1)^{m-2-k}(n+k-1)(s+1)^{k-1} + (n-1)(-1)^{m-2}\\
    &= (s+1)p_{m-1, n+1} +  (n-1)(-1)^{m-2}. 
\end{align*}
By assumption, $p_{m-1, n+1}$ is log-concave with no internal zeros, hence so is $(s+1)p_{m-1, n+1}$.  We denote by $a_i$ the coefficient of the degree $i$ term of $p_{m, n}$. It suffices to check that $a_1^2 \geq a_0a_2$ and $a_1 >0$. 

When $m$ is odd (hence assuming $m\geq 3$),
\begin{equation*}
\begin{cases}
a_0 = \frac{m+1}{2}\\
a_1 = \frac{(m-1)(n+m)}{2}\\
a_2 = \frac{(m-1)(2nm+2m^2-6n-7m+7)}{8}
\end{cases}
\end{equation*}
and
\begin{equation*}
    a_1^2 - a_0a_2 = \frac{(m-1)(4n^2m+6nm^2+2m^3-4n^2-12nm-7m^2+14n+12m-11)}{16} > 0.
\end{equation*}
When $m$ is even,
\begin{equation*}
\begin{cases}
a_0 = \frac{m+2n-2}{2}\\
 a_1 = \frac{nm+m^2-2n-2m+2}{8}\\
 a_2 = \frac{(m-2)(2nm+2m^2-4n-5m+4)}{8}
\end{cases}
\end{equation*}
and
\begin{equation*}
    a_1^2-a_0a_2 = \frac{m(2nm^2+2m^3-2nm-3m^2-4n+4)}{16} > 0.
\end{equation*}
We see in both cases that $a_1^2>a_0a_2$ and $a_1>0$, completing the proof. 
\end{proof}

\section{Appendix: Coefficient Calculations}
We present explicit expressions for the coefficients of the recursive polynomials used to prove the results of Section \ref{sec:recursive_polynomials} here. The open source mathematics software Sage (\cite{Sage}) was used to simplify some algebraic expressions. As the computations are straightforward, we state only the resulting simplified expressions. Throughout the following, unless otherwise stated, $a_i$ denotes the coefficient of the degree $i$ term of the polynomial in question. It is assumed in the expression for $a_j$ that the polynomial has degree at least $j$ ($a_j=0$ otherwise). Further, when checking either the log-concave (LC) inequality (\ref{eq:LC_ineq}) or the ultra log-concave (ULC) inequality (\ref{eq:ULC_factor}) for degree $j$, we may assume that the polynomial has degree at least $j+1$ as otherwise the inequalities are both trivially satisfied by $a_{j+1}=0$. 
Note that for $m\geq 1$, $f_m, g_{m,n}, h_m$ are all of degree $m-1$ while $b_{m,n}$ is of degree $m$. 

\subsection{Coefficients and inequality computations for $f_m$}
\label{sec:f_m_coeff}
We compute the coefficients for $f_m$ using the explicit formula for $f_m$ given by Claim 3.5 of \cite{MR4620363}
\begin{equation}
    f_m(s) = \sum_{j=1}^{n-1}\sum_{k=1}^{\lfloor \frac{m}{2}\rfloor} {m-2k \choose j-1} s^j + \begin{cases}
    1 & m \text{ is odd}\\
    0 & m \text{ is even}
    \end{cases}
\end{equation}
together with the well-known formulas for the sum of the first $n$ positive (even, odd) integers, squares, and cubes. 
We obtain
\begin{equation}
\label{eq:f_m_coeff_deg0}
    a_0=\begin{cases}
    1 & m \text{ is odd }\\
    0 & m \text{ is even}
    \end{cases},
\end{equation}
\begin{equation}
\label{eq:f_m_coeff_deg1}  
a_1 = \begin{cases}
\frac{m-1}{2} & m \text{ is odd}\\
\frac{m}{2} & m \text{ is even},
\end{cases}
\end{equation}
\begin{equation}
\label{eq:f_m_coeff_deg2}
a_2 = \begin{cases}
 \frac{(m-1)^2}{4} & m \text{ is odd}\\
 \frac{m(m-2)}{4} & m \text{ is even}
\end{cases},
\end{equation}
\begin{equation}
\label{eq:f_m_coeff_deg3}
a_3=
\begin{cases}
\frac{(m-1)(m-3)(2m-1)}{24} & m \text{ is odd}\\
\frac{m(m-2)(2m-5)}{24} & m \text{ is even}
\end{cases},
\end{equation}
and
\begin{equation}
\label{eq:f_m_coeff_deg4}
a_4 = \begin{cases}
    \frac{(m-1)(m-3)(m^2-4m+1)}{48} & m \text{ is odd}\\
    \frac{m(m-2)^2(m-4)}{48} & m \text{ is even}
\end{cases}. 
\end{equation}
For completion, we prove (\ref{eq:f_m_coeff_deg0})-(\ref{eq:f_m_coeff_deg4}) by induction using the recursive relation (\ref{eq:f_m+1_from_f_m}). Note that (\ref{eq:f_m_coeff_deg0})-(\ref{eq:f_m_coeff_deg4}) hold for $f_0=0, f_1=1, f_2=s$. Assume $m\geq 3$ and that (\ref{eq:f_m_coeff_deg0})-(\ref{eq:f_m_coeff_deg4}) hold for $f_{m-1}$. 

We consider first the case where $m$ is odd. Then, 
\begin{align*}
    f_{m-1}& = \text{(higher order terms)} + \frac{(m-1)(m-3)^2(m-5)}{48}s^4 \\
    &+ \frac{(m-1)(m-3)(2m-7)}{24}s^3 + \frac{(m-1)(m-3)}{4}s^2 + \frac{m-1}{2}s
\end{align*}
hence
\begin{equation*}
    a_4 = \frac{(m-1)(m-3)^2(m-5)}{48} + \frac{(m-1)(m-3)(2m-7)}{24}
    = \frac{(m-1)(m-3)(m^2-4m+1)}{48},
\end{equation*}
\begin{equation*}
    a_3 = \frac{(m-1)(m-3)(2m-7)}{24} + \frac{(m-1)(m-3)}{4}
    = \frac{(m-1)(m-3)(2m-1)}{24},
\end{equation*}
\begin{equation*}
    a_2 = \frac{(m-1)(m-3)}{4} + \frac{m-1}{2} = \frac{(m-1)^2}{4},
\end{equation*}
\begin{equation*}
    a_1 = \frac{m-1}{2},
\end{equation*}
and 
\begin{equation*}
        a_0 = (-1)^{m-1} = 1.
\end{equation*}
Suppose now that $m$ is even. Then,
\begin{align*}
    f_{m-1}& = \text{(higher order terms)} + \frac{(m-2)(m-4)(m^2-6m+6)}{48}s^4 \\
    & + \frac{(m-2)(m-4)(2m-3)}{24}s^3 + \frac{(m-2)^2}{4}s^2 + \frac{m-2}{2}s + 1,
\end{align*}
hence
\begin{equation*}
    a_4 = \frac{(m-2)(m-4)(m^2-6m+6)}{48} + \frac{(m-2)(m-4)(2m-3)}{24}
    = \frac{m(m-2)^2(m-4)}{48},
\end{equation*}
\begin{equation*}
    a_3 = \frac{(m-2)(m-4)(2m-3)}{24} + \frac{(m-2)^2}{4}
    = \frac{m(m-2)(2m-5)}{24},
\end{equation*}
\begin{equation*}
    a_2 = \frac{(m-2)^2}{4} + \frac{m-2}{2}
    = \frac{(m-2)(m-4)}{4},
\end{equation*}
\begin{equation*}
    a_1 = \frac{m-2}{2} + 1=\frac{m}{2},
\end{equation*}
and
\begin{equation*}
    a_0 = 1 + (-1)^{m-1} = 0. 
\end{equation*}
Having established (\ref{eq:f_m_coeff_deg0})-(\ref{eq:f_m_coeff_deg4}), we now check the inequalities for log-concavity and ultra log-concavity. 

We begin by checking log-concavity (\ref{eq:LC_ineq}) for the degree 1 coefficient for the proof of Lemma \ref{lem:f_m_LC}. We compute
\begin{equation*}
   a_1^2 - a_0a_2 = \begin{cases}
        \frac{(m-1)^2}{4} - \frac{(m-1)^2}{4} = 0 & m \text{ is odd }\\
         a_1^2 - a_0a_2 = \frac{m^2}{4} > 0 & m \text{ is even}
   \end{cases}. 
\end{equation*}
We now check the $\text{ULC}$ inequality (\ref{eq:ULC_factor}) for degrees 1, 2, and 3 for the proof of Proposition \ref{prop:f_g_h_b_ULC}. Recall that $f_m$ has degree $m-1$. Assuming $m\geq 4$, we compute
\begin{equation*}
   \text{Degree 1:} \quad  (m-2)a_1^2 - 2(m-1)a_0a_2 = \begin{cases}
    -\frac{m(m-1)^2}{4} < 0 & m \text{ is odd}\\
    \frac{(m-2)m^2}{4}> 0 & m \text{ is even}
    \end{cases},
\end{equation*}
\begin{equation*}
    \text{Degree 2:} \quad  2(m-3)a_2^2-3(m-2)a_1a_3 =
    \begin{cases}
    \frac{m(m-1)^2(m-3)}{16} > 0 & m \text{ is odd} \\
   - \frac{m^2(m-2)^2}{16}<0 & m \text{ is even}
    \end{cases},
\end{equation*}
and
\begin{equation*}
     \text{Degree 3:} \quad 3(m-4)a_3^2 - 4(m-3)a_2a_4 = 
    \begin{cases}
       -\frac{m(m-1)^2(m-3)^2}{64} <0 & m \text{ is odd}\\
       \frac{m^2(m-2)^2(m-4)}{192} \geq 0 & m \text{ is even}
    \end{cases}. 
\end{equation*}
Finally, we check the $\text{ULC}(m)$ inequality (\ref{eq:ULC_factor_orderm}) for degrees 1, 2, and 3 for the proof of Proposition \ref{prop:f_g_h_b_ULC_m}.
Assuming $m\geq 6$, we compute
\begin{equation*}
     \text{Degree 1:} \quad  (m-1)a_1^2 - 2ma_0a_2 =
     \begin{cases}
       -\frac{(m+1)(m-1)^2}{4} < 0 & m \text{ is odd}\\
       \frac{m^2(m-1)}{4} >0 & m \text{ is even}
     \end{cases},
\end{equation*}
\begin{equation*}
     \text{Degree 2:} \quad 
      2(m-2)a_2^2 - 3(m-1)a_1a_3
      = \begin{cases}
        \frac{(m+1)(m-1)^3}{16} > 0 & m \text{ is odd}\\
        -\frac{m^2(m-2)(m-3)}{16} <0 & m \text{ is even}
      \end{cases},
\end{equation*}
and
\begin{equation*}
     \text{Degree 3:} \quad 
      3(m-3)a_3^2 - 4(m-2)a_2a_4 = \begin{cases}
       \frac{(m+1)^2(m-1)^2(m-3)}{192} > 0 & m \text{ is odd}\\
       \frac{m^2(m-2)^2(5m-11)}{192}>0 & m \text{ is even}
      \end{cases}. 
\end{equation*}
We summarize the results of these calculations below; for simplicity, assume $m\geq 6$. 
\begin{table}[H]
    \centering
    \begin{tabular}{|c|c|c|c|}
    \hline 
        Parity of $m$ &  Degree  & ULC & ULC$(m)$\\
        \hline 
        Odd & 1 & Fails & Satisfies\\
        Odd & 2 & Satisfies & Satisfies\\
        Odd & 3 & Fails & Satisfies \\
        Even & 1 & Satisfies & Satisfies\\
        Even & 2 & Fails & Fails\\
        Even & 3 & Satisfies & Satisfies\\
        \hline 
    \end{tabular}
    \caption{Checking ULC and ULC$(m)$ for $f_m$ for degrees 1, 2 and 3}
    \label{tab:f_m_ineq_check}
\end{table}

\subsection{Coefficients and inequality computations for $g_{m,n}$}
\label{sec:g_m_n_coeff}
Using the equation
\begin{equation*}
    g_{m,n}(s) = n(s+1)^{m-1} - f_m(s)
\end{equation*}
and the coefficients of $f_m(s)$ as calculated in (\ref{eq:f_m_coeff_deg0})-(\ref{eq:f_m_coeff_deg4}), we obtain the following expressions for the coefficients of $g_{m,n}$:
\begin{equation}
\label{eq:g_mn_coeff_deg0}
a_0 = \begin{cases}
    n-1 & m \text{ is odd}\\
    n & m \text{ is even}
\end{cases},
\end{equation}
\begin{equation}
\label{eq:g_mn_coeff_deg1}
a_1 = \begin{cases}
  \frac{(m-1)(2n-1)}{2}& m \text{ is odd}\\
  \frac{2n(m-1)-m}{2} & m \text{ is even }
\end{cases},
\end{equation}
\begin{equation}
\label{eq:g_mn_coeff_deg2}
a_2 = \begin{cases}
 \frac{(m-1)[(m-1)(2n-1)-2n]}{4} & m \text{ is odd}\\
 \frac{(m-2)(2n(m-1)-m)}{4} & m \text{ is even}
\end{cases},
\end{equation}
\begin{equation}
\label{eq:g_mn_coeff_deg3}
a_3 = \begin{cases}
 \frac{(m-1)(m-3)(4n(m-2)-(2m-1))}{24} & m \text{ is odd}\\
 \frac{(m-2)(4n(m-1)(m-3) - m(2m-5))}{24} & m \text{ is even }
\end{cases},
\end{equation}
and
\begin{equation}
 \label{eq:g_mn_coeff_deg4}
a_4 = \begin{cases}
  \frac{(m-1)(m-3)(2n(m-2)(m-4)-(m^2-4m+1))}{48} & m \text{ is odd}\\
   \frac{(m-2)(m-4)(2n(m-1)(m-3) - m(m-2))}{48} & m \text{ is even}
\end{cases}. 
\end{equation}
For $g_{m,m}$, these simplify to 
\begin{equation}
\label{eq:g_m_coeff_deg0}
a_0 = \begin{cases}
    m-1 & m \text{ is odd}\\
    m & m \text{ is even}
\end{cases},
\end{equation}
\begin{equation}
\label{eq:g_m_coeff_deg1}
a_1 = \begin{cases}
  \frac{(m-1)(2m-1)}{2}& m \text{ is odd}\\
  \frac{m(2m-3)}{2} & m \text{ is even }
\end{cases},
\end{equation}
\begin{equation}
\label{eq:g_m_coeff_deg2}
a_2 = \begin{cases}
 \frac{(m-1)(2m^2-5m+1)}{4} & m \text{ is odd}\\
 \frac{m(m-2)(2m-3)}{4} & m \text{ is even}
\end{cases},
\end{equation}
\begin{equation}
\label{eq:g_m_coeff_deg3}
a_3 = \begin{cases}
 \frac{(m-1)(m-3)(4m^2-10m+1)}{24} & m \text{ is odd}\\
 \frac{m(m-2)(4m^2-18m+17)}{24} & m \text{ is even }
\end{cases},
\end{equation}
and
\begin{equation}
 \label{eq:g_m_coeff_deg4}
a_4 = \begin{cases}
  \frac{(m-1)(m-3)(2m^3-13m^2+20m-1)}{48} & m \text{ is odd}\\
   \frac{m(m-2)(m-4)(2m^2-9m+8)}{48} & m \text{ is even}
\end{cases}. 
\end{equation}
Using (\ref{eq:g_mn_coeff_deg0})-(\ref{eq:g_mn_coeff_deg2}), we check the log-concavity inequality (\ref{eq:LC_ineq}) in degree 1 for the proof of Lemma \ref{lem:g_m_LC}. Assuming $m\geq 3, n\geq 1$, we compute
\begin{equation*}
 a_1^2 - a_0a_2 = \begin{cases}
  \frac{(m-1)n(2nm-m-1)}{4} > 0 & m \text{ is odd}\\
  \frac{m(n-1)(2nm-2n-m)}{4} > 0 & m \text{ is even}
 \end{cases}.    
\end{equation*}
We now check the ULC inequality (\ref{eq:ULC_factor}) for the coefficients of degrees 1, 2, and 3 for the proofs of Lemma \ref{lem:g_m_n_ULC_3} and Proposition \ref{prop:f_g_h_b_ULC}. Recall that $g_{m,n}$ has degree $m-1$. Assuming $m\geq 4, n\geq 1$, we obtain
\begin{equation*}
    \text{Degree 1}: \quad 
    (m-2)a_1^2 - 2(m-1)a_0a_2= \begin{cases}
    \frac{(m-1)^2(2nm-2n-m)}{4} > 0 & m \text{ is odd}\\
    -\frac{m(m-2)(2nm-2n-m)}{4} < 0 & m \text{ is even}
    \end{cases},
\end{equation*}
\begin{equation*}
    \text{Degree 2}: \quad  2(m-3)a_2^2- 3(m-2)a_1a_3 = \begin{cases}
      -\frac{(m-1)^2(m-3)^2(2nm-4n-m)}{16} < 0 & m \text{ is odd}\\
      \frac{m(m-2)^2(2nm-2n-m)}{16} > 0 & m \text{ is even}
    \end{cases},
\end{equation*}
and
\begin{equation*}
    \text{Degree 3}: \quad \begin{cases}
     \frac{(m-1)^2(m-3)^2(8nm-16n-3m)}{192} > 0 & m \text{ is odd}\\
     \frac{m^2(m-2)^2(m-4)}{192} > 0 & m \text{ is even}
    \end{cases}. 
\end{equation*}
Finally, we check the $\text{ULC}(m)$ inequality (\ref{eq:ULC_factor_orderm}) for degrees 1, 2, and 3 for the proof of Proposition \ref{prop:f_g_h_b_ULC_m}. Note that for Proposition  \ref{prop:f_g_h_b_ULC_m}, it suffices to check the inequalities for $g_{m,m}$ specifically rather than $g_{m,n}$ more generally. For completeness, we include the expressions for the left-hand sides of (\ref{eq:ULC_factor}) for all combinations of $m,n$, but state only the result of the inequality for $m=n$ when the inequality is dependent on the relative sizes of $m,n$.

Assuming $m\geq 6$, we compute
\begin{equation*}
    \text{Degree 1}: \quad 
    (m-1)a_1^2 - 2ma_0a_2 = \begin{cases}
   \frac{(m-1)(2nm^2+4n^2-2nm-m^2-4n+1)}{4} > 0 & m \text{ is odd}\\
   -\frac{(m^2-2n-m)(2nm-2n-m)}{4} & m \text{ is even}\\
   -\frac{m^2(m-3)(2m-3)}{4} < 0 & m=n \text{ is even}
    \end{cases},
\end{equation*}
\begin{flalign*}
 & \text{Degree 2}: \quad     2(m-2)a_2^2 - 3(m-1)a_1a_3  \\ 
 &= 
 \begin{cases}
  \frac{(m-1)^2(8n^2m-2nm^2-16n^2+m^2+2n-1)}{16} > 0 & m \text{ is odd}  \\
  \frac{(m-2)(2nm-2n-m)(4nm+m^2-4n-3m)}{16} > 0 & m \text{ is even}
 \end{cases},
\end{flalign*}
and
\begin{flalign*}
  &  \text{Degree 3}:  \quad 
     3(m-3)a_3^2 - 4(m-2)a_2a_4 \\
& = \begin{cases}
       \frac{(m-1)^2(m-3)(4nm-8n-m-1)^2}{192} > 0 & m \text{ is odd}  \\
       \frac{(m-2)^2(16n^2m^3 - 80n^2m^2 - 16nm^3 + 112n^2m + 56nm^2 + 5m^3 - 48n^2 - 40nm - 11m^2)}{192} \geq 0 & m \text{ is even}
     \end{cases}   
\end{flalign*}
as \begin{equation*}
    16n^2m^3 - 80n^2m^2 - 16nm^3 + 56nm^2
    = 8nm^2(2nm - 10n - 2m + 7) > 0
\end{equation*}
for $m\geq 6, n\geq 1$. 
We summarize the results of these calculations below; for simplicity, assume $m\geq 6$ and $n\geq 1$.  
\begin{table}[H]
    \centering
    \begin{tabular}{|c|c|c|c|}
    \hline 
        Parity of $m$ &  Degree  & ULC & ULC$(m)$\\
        \hline 
        Odd & 1 & Satisfies & Satisfies\\
        Odd & 2 & Fails & Satisfies\\
        Odd & 3 & Satisfies & Satisfies \\
        Even & 1 & Fails & Fails for $n=m$\\
        Even & 2 & Satisfies & Satisfies\\
        Even & 3 & Satisfies & Satisfies\\
        \hline 
    \end{tabular}
    \caption{Checking ULC and ULC$(m)$ for $g_{m,n}$ for degrees 1, 2 and 3, assuming $m\geq 6, n\geq 1$. }
    \label{tab:g_mn_ineq_check}
\end{table}

\subsection{Coefficients and inequality computations for $b_{m,n}$}
\label{sec:b_m_n_coeff}
Substituting the coefficients of $f_m(s)$ as calculated in (\ref{eq:f_m_coeff_deg0})-(\ref{eq:f_m_coeff_deg4}) into the equation
\begin{equation*}
    b_{m,n}(s) = n(s+1)^{m-1} + (s+1)f_m,
\end{equation*}
 we calculate the following expressions for the coefficients of $b_{m,n}$: 
\begin{equation}
\label{eq:b_mn_coeff_deg0}
 a_0 = \begin{cases}
  n+1 & m \text{ is odd}\\
  n & m \text{ is even}
 \end{cases} ,  
\end{equation}
\begin{equation}
\label{eq:b_mn_coeff_deg1}
    a_1 = \begin{cases}
       \frac{(m-1)(2n+1)+2}{2} & m \text{ is odd}  \\
        \frac{2n(m-1)+m}{2} & m \text{ is even}
    \end{cases},
\end{equation}
\begin{equation}
\label{eq:b_mn_coeff_deg2}
    a_2 = \begin{cases}
    \frac{(m-1)(2n(m-2)+m+1)}{4} & m \text{ is odd}  \\  
    \frac{2n(m-1)(m-2)+m^2}{4} & m \text{ is even}
    \end{cases},
\end{equation}
\begin{equation}
\label{eq:b_mn_coeff_deg3}    
a_3 = \begin{cases}
    \frac{(m-1)(4n(m-2)(m-3)+2m^2-m-3)}{24} & m \text{ is odd}\\
    \frac{(m-2)(4n(m-1)(m-3)+m(2m+1))}{24} & m \text{ is even}
\end{cases},
\end{equation}
\begin{equation}
 \label{eq:b_mn_coeff_deg4}
 a_4 = \begin{cases}
  \frac{(m-1)(m-3)(2n(m-2)(m-4) + m^2-1)}{48} & m \text{ is odd}   \\
  \frac{(m-2)(2n(m-1)(m-3)(m-4) + m(m^2-2m-2))}{48} & m \text{ is even}
 \end{cases}
\end{equation}
We use (\ref{eq:b_mn_coeff_deg0})-(\ref{eq:b_mn_coeff_deg2}), we check the LC inequality (\ref{eq:LC_ineq}) in degree 1 and 2 for the proof of Lemma \ref{lem:b_m_LC}. Assuming $m\geq 3, n\geq 1$, we compute
\begin{equation*}
   \text{Degree 1}: \quad  a_1^2-a_0a_2 = \begin{cases}
     \frac{2n^2m^2-2n^2m+nm^2 + 6nm-7n+2m+2}{4} > 0 & m \text{ is odd} \\
     \frac{m(2n^2m-2n^2+3nm-4n+m)}{4} > 0 & m \text{ is even}
    \end{cases}
\end{equation*}
\begin{align*}
    & \text{Degree 2}: \quad   a_2^2 - a_1a_3 \\
    & =
    \begin{cases}
        \frac{(m-1)(4n^2m^3-12n^2m^2+4nm^3 + 8n^2m-2nm^2 + m^3 - 12nm + 2m^2 - 6n+m)}{48} > 0 & m \text{ is odd} \\
        \frac{m(4n^2m^3-16n^2m^2+4nm^3+20n^2m-2nm^2+m^3-8n^2-22nm+3m^2+20n+2m)}{48} > 0 & m \text{ is even}
    \end{cases}
\end{align*}
For the odd case, we use the observations that for $m=3$,
\begin{equation*}
    a_2^2 - a_1a_3 = \frac{6n^2+12n+1}{12} > 0 
\end{equation*}
and for $m\geq 5$, we have the inequalities
\begin{equation*}
    \begin{cases}
     4n^2m^3 \geq 12n^2m^2\\
     nm^3 > 2nm^2\\
     3nm^3 \geq 12nm+6n=6n(2m+1)\\
    \end{cases}. 
\end{equation*}
For the even case, we use the inequalities
\begin{equation*}
    \begin{cases}
        4n^2m^3-16n^2m^2+20n^2m-8n^2 = 4n^2(m^3-4m^2+5m-2) > 0\\
        4nm^3-2nm^2-22nm+20 = 2n(2m^3-m^2-11m+10) > 0
    \end{cases}
\end{equation*}
that hold for $m\geq 4, n\geq 1$. 

We now check the ULC inequality (\ref{eq:ULC_factor}) of the coefficients of degree 1, 2, and 3 for the proofs of Lemma \ref{lem:b_m_n_ULC_3} and Propositions \ref{prop:f_g_h_b_ULC}, \ref{prop:f_g_h_b_ULC_m}. Note that $b_{m,n}$ is of degree $m$. As in Section \ref{sec:g_m_n_coeff},  for Propositions  \ref{prop:f_g_h_b_ULC}, \ref{prop:f_g_h_b_ULC_m}, it suffices to check the inequalities for $b_{m,m}$ specifically, but we include the expressions for the left-hand sides of (\ref{eq:ULC_factor}) for all combinations of $m,n$ more generally and the state the inequality result in general when possible. Assuming $m\geq 4, n\geq 2$, we obtain
\begin{equation*}
    \text{Degree 1}: \quad  (m-1)a_1^2 - 2ma_0a_2 =
    \begin{cases}
      -\frac{(m-1)(2nm^2-4n^2-6nm+m^2+4n-1)}{4} &  m \text{ is odd}   \\
      - \frac{(m-1)(2m^3-9m^2+4m-1)}{4} < 0 & m=n \text{ is odd}\\
       \frac{2nm^3+4n^2m-8nm^2+m^3-4n^2 + 4nm-m^2}{4} >0 & m \text{ is even}
    \end{cases},
\end{equation*}
\begin{align*}
    & \text{Degree 2}: \quad 
     2(m-2)a_2^2 - 3(m-1)a_1a_3 \\
     & = 
     \begin{cases}
       \frac{(m-1)^2(8n^2m-2nm^2-16n^2-m^2+2n-2m-1)}{16} \geq 0 & m\leq n, m \text{ is odd}   \\
        \frac{(m-2)(8n^2m^2+ 2nm^3-16n^2m-12nm^2 + m^3+ 8n^2+ 10nm+m^2)}{16} > 0 & m \text{ is even}
     \end{cases}
\end{align*}
\begin{align*}
    & \text{Degree 3}: \quad
 3(m-3)a_3^2 - 4(m-2)a_2a_4 \\
 &= \begin{cases}
   \frac{(m-1)^2(m-3)(4nm-8n-m-1)^2}{192} \geq 0 & m \text{ is odd}   \\
   \frac{(m-2)^2(16n^2m^3 - 80n^2m^2 - 16nm^3 + 112n^2m + 56nm^2-3m^3-48n^2-40nm-3m^2)}{192} & m \text{ is even}
 \end{cases}
\end{align*}
In the even degree 3 case, we make use of the following explicit computations:
\begin{equation*}
3(m-3)a_3^2 - 4(m-2)a_2a_4 = 
    \begin{cases}
    3n^2-6n-5 > 0 & m=4, n\geq 3\\
     100n^2 - 140n - 63 >0   & m =6, n\geq 2\\
     12(4655n^2+2898n-27) > 0 & m=8, n\geq 1\\
     4(756n^2-900n-275)> 0& m=10, n\geq 2\\
    \end{cases}
\end{equation*}
and that for $m\geq 12, n\geq 2$, the inequalities
\begin{equation*}
     112n^2m + 56nm^2-48n^2-40nm-3m^2 > 0
\end{equation*}
and 
\begin{equation*}
        16n^2m^3 - 80n^2m^2 - 16nm^3 - 3m^3
    = 7n^2m^3 - 80n^2m^2 + 8n^2m^3-16nm^3 + n^2m^3 - 3m^3
    > 0 
\end{equation*}
hold. 
We summarize the results of these computations below: 
\begin{table}[H]
    \centering
    \begin{tabular}{|c|c|c|}
    \hline 
        Parity of $m$ &  Degree  & ULC \\
        \hline 
        Odd & 1 & Fails for $m=n$ \\
        Odd & 2 & Satisfies for $m\geq n$\\
        Odd & 3 & Satisfies \\
        Even & 1 & Satisfies\\
        Even & 2 & Satisfies \\
        Even & 3 & Satisfies \\
        \hline 
    \end{tabular}
    \caption{Checking ULC (equivalent to ULC$m$) for $b_{m,n}$ for degrees 1, 2 and 3, assuming $m\geq 4, n\geq 3$. }
    \label{tab:b_mn_ineq_check}
\end{table}



\bibliographystyle{alpha}
\bibliography{BIB}

\begin{thebibliography}{BGRR11}

\bibitem[AHK18]{Hodge_Combo_Adiprasito_Huh_Katz}
Karim Adiprasito, June Huh, and Eric Katz.
\newblock Hodge theory for combinatorial geometries.
\newblock {\em Annals of Mathematics}, 188(2):381--452, 2018.

\bibitem[AM09]{Banana}
Paolo Aluffi and Matilde Marcolli.
\newblock Feynman motives of banana graphs.
\newblock {\em Commun. Number Theory Phys.}, 3(1):1--57, 2009.

\bibitem[AM11a]{Algebro_Geometric_Feynman_Rules}
Paolo Aluffi and Matilde Marcolli.
\newblock Algebro-geometric {F}eynman rules.
\newblock {\em International Journal of Geometric Methods in Modern Physics}, 8(1):203--237, 2011.

\bibitem[AM11b]{Deletion_Contraction}
Paolo Aluffi and Matilde Marcolli.
\newblock Feynman motives and deletion contraction relations.
\newblock {\em Contemporary Mathematics}, 538, 2011.

\bibitem[AMQ23]{MR4620363}
Paolo Aluffi, Matilde Marcolli, and Waleed Qaisar.
\newblock Motives of melonic graphs.
\newblock {\em Ann. Inst. Henri Poincar\'{e} D}, 10(3):503--554, 2023.

\bibitem[Bak18]{Hodge_Combo_Baker}
Matt Baker.
\newblock Hodge theory in combinatorics.
\newblock {\em Bulletin of the American Math Society}, 55(1):57--80, January 2018.

\bibitem[BGRR11]{Bonzom_2011}
Valentin Bonzom, Razvan Gurau, Aldo Riello, and Vincent Rivasseau.
\newblock Critical behavior of colored tensor models in the large n limit.
\newblock {\em Nuclear Physics B}, 853(1):174–195, December 2011.

\bibitem[BH20]{BrandenHuh_Lorentzian}
Petter {Br{\"{a}}nd{\'{e}}n} and June Huh.
\newblock Lorentzian polynomials.
\newblock {\em Annals of Mathematics}, 192(3):821--891, 2020.

\bibitem[Br{\"{a}}14]{branden2014unimodalitylogconcavityrealrootedness}
Petter Br{\"{a}}nd{\'{e}}n.
\newblock Unimodality, log-concavity, real-rootedness and beyond, 2014.
\newblock [arXiv:1410.6601].

\bibitem[Bre94]{Brenti_LG_survey}
Francesco Brenti.
\newblock Log-concave and unimodal sequences in algebra, combintorics, and geometry: an update.
\newblock In {\em Jerusalem combinatorics '93}, volume 178, pages 71--89. American Math Society, 1994.

\bibitem[Hog74]{Hoggar_Chromative_Poly_Log_Concavity}
S.~G. Hoggar.
\newblock Chromatic polynomials and logarithmic concavity.
\newblock {\em Journal of Combinatorial Theory}, 16:248--254, 1974.

\bibitem[Koo06]{WoongKook_ProductLogConcavePoly}
Woong Kook.
\newblock On the product of log concave polynomials.
\newblock {\em Integers: Electronic Journal of Combinatorial Number Theory}, 6, 2006.

\bibitem[KT17]{Klebanov_2017}
Igor~R. Klebanov and Grigory Tarnopolsky.
\newblock Uncolored random tensors, melon diagrams, and the {S}achdev-{Y}e-{K}itaev models.
\newblock {\em Physical Review D}, 95(4), February 2017.

\bibitem[Lig97]{Liggett_ULC_NegDependence_ConvoProd}
Thomas~M. Liggett.
\newblock Ultra logconcave sequences and negative dependence.
\newblock {\em Journal of Combinatorial Theory}, 79:315--325, 1997.

\bibitem[Mar10a]{marcolli2009feynmanintegralsmotives}
Matilde Marcolli.
\newblock Feynman integrals and motives.
\newblock In {\em European {C}ongress of {M}athematics}, pages 293--332. Eur. Math. Soc., Z\"{u}rich, 2010.

\bibitem[Mar10b]{Feyman_Motives_Book}
Matilde Marcolli.
\newblock {\em Feynman Motives}.
\newblock World Scientific, 2010.

\bibitem[Sag]{Sage}
{SageMath}.
\newblock \url{https://www.sagemath.org/}.
\newblock Accessed: 2023-07-05.

\bibitem[Sta06]{Stanley_LogConcave_Survey}
Richard Stanley.
\newblock Log-concave and unimodal sequences in algebra, combinatorics, and geometry.
\newblock {\em Annals of the New York Academy of Sciences}, 576(1):500--535, 2006.

\end{thebibliography}

\end{document}